\documentclass{amsart} 

\usepackage{amsmath,amsthm}     
\usepackage{amssymb}            
\usepackage{euscript}           
\usepackage{graphicx,enumerate,calc,lscape,color}
\definecolor{red}{rgb}{1,0,0}
\usepackage[matrix,arrow,curve,frame]{xy}    

\xymatrixcolsep{1.9pc}                          
\xymatrixrowsep{1.9pc}
\newdir{ >}{{}*!/-5pt/\dir{>}}                  

\newtheorem{thm}[subsection]{Theorem}
\newtheorem{defn}[subsection]{Definition}
\newtheorem{prop}[subsection]{Proposition}

\newtheorem{cor}[subsection]{Corollary}
\newtheorem{lemma}[subsection]{Lemma}

\newtheorem{ques}[subsection]{Problem}

\theoremstyle{definition}  

\newtheorem{remark}[subsection]{Remark}

  {\end{list}}

\newcommand{\dfn}{\textbf} 

\newcommand{\mdfn}[1]{\dfn{\mathversion{bold}#1}} 

\newcommand{\Smash}             {\wedge}

\newcommand{\Wedge}             {\vee}

\newcommand{\tens}              {\otimes}               
\newcommand{\iso}               {\cong}

\newcommand{\cat}{\EuScript}    
\newcommand{\cC}{{\cat C}}
\newcommand{\cD}{{\cat D}}

\newcommand{\field}[1]  {\mathbb #1} 
\newcommand{\A}         {\field A}

\newcommand{\Z}         {\field Z}
\newcommand{\C}         {\field C}
\newcommand{\M}         {\field M}

\DeclareMathOperator*{\hocolim}{hocolim}

\DeclareMathOperator{\Spec}{Spec}
\DeclareMathOperator{\Hom}{Hom}

\DeclareMathOperator{\Tot}{Tot}

\DeclareMathOperator{\Ext}{Ext}

\newcommand{\ra}{\rightarrow}                   
\newcommand{\lra}{\longrightarrow}              
\newcommand{\la}{\leftarrow}                    
\newcommand{\llra}[1]{\stackrel{#1}{\lra}}      

\newcommand{\blank}{-}                          

\newcommand{\he}{\simeq}

\newcommand{\rea}[1]{|{#1}|}             
\newcommand{\map}{\rightarrow}

\newcommand{\ceck}[1]{\Cech(#1)}         
\newcommand{\oceck}[1]{\Cech^{o}(#1)}    
\newcommand{\oreal}[1]{\rea{\oceck{U}}}  
\newcommand{\creal}[1]{\rea{\ceck{U}}}   

\newcommand{\Cech}{\check{C}}

\DeclareMathOperator{\Gr}{Gr}

\newcommand{\F}{\mathbb{F}}

\newcommand{\cl}{\mathrm{cl}}
\numberwithin{equation}{subsection}

\newcommand{\bH}{\overline{H}}
\newcommand{\bE}{\overline{E}}
\newcommand{\MGL}{MGL}
\newcommand{\BPL}{BPL}
\newcommand{\tmf}{tmf}

\newenvironment{myequation}
  {\addtocounter{subsection}{1}\begin{eqnarray}}
  {\end{eqnarray}$\!\!$}

\DeclareMathOperator{\Sq}{Sq}

\newcommand{\bbox}{\circle{0.4}}

\newcommand{\bc}{\circle*{0.3}}
\newcommand{\rc}{{\color[rgb]{1,0,0}\circle*{0.3}}}

\newcommand{\bigrc}{{\color[rgb]{1,0,0}\circle{0.4}}}
\newcommand{\smallrc}{{\color[rgb]{1,0,0}\circle*{0.2}}}

\newcommand{\vl}{{\line(0,1){1}}}
\newcommand{\vlvar}{{{\color[rgb]{0,0,1}\line(0,1){1}}}}
\newcommand{\vlr}{{\line(1,4){0.25}}}
\newcommand{\vll}{{\line(-1,4){0.25}}}
\newcommand{\vllvar}{{{\color[rgb]{0,0,1}\line(-1,4){0.25}}}}

\newcommand{\vltower}{{\vector(0,1){1}}}

\newcommand{\dltower}{{{\color[rgb]{1,0,0} \vector(1,1){1}}}}

\newcommand{\da}{\vector(1,1){0.8}}

\newcommand{\dl}{{\line(1,1){1}}}
\newcommand{\rdl}{{{\color[rgb]{1,0,0} \line(1,1){1}}}}
\newcommand{\rdlr}{{{\color[rgb]{1,0,0} \line(5,4){1.25}}}}

\newcommand{\htwo}{{{\color[rgb]{0,0.7,0} \line(3,1){3}}}}
\newcommand{\htwovar}{{{\color[rgb]{1,0,1} \line(3,1){3}}}}
\newcommand{\htwoshort}{{{\color[rgb]{0,0.7,0} \line(3,1){2.75}}}}
\newcommand{\htwodshort}{{{\color[rgb]{0,0.7,0} \line(3,1){2.5}}}}
\newcommand{\htwoshortvar}{{{\color[rgb]{1,0,1} \line(3,1){2.75}}}}

\newcommand{\done}{{{\color[rgb]{0,0,1} \vector(-1,1){1}}}}

\newcommand{\dthree}{{{\color[rgb]{0,0.7,0.7} \vector(-1,3){1}}}}
\newcommand{\tinydthree}{{{\color[rgb]{0,0.7,0.7} \vector(-1,3){0.3}}}}

\newcommand{\dfive}{{{\color[rgb]{0.2,0.2,0.7} \vector(-1,4){1}}}}

\newcommand{\SSH}{\langle S\rangle_H}
\renewcommand{\SS}{\langle S\rangle}
\newcommand{\Abar}{\bar{A}}

\begin{document}

\title{The motivic Adams spectral sequence}

\author{Daniel Dugger}
\author{Daniel C. Isaksen}

\address{Department of Mathematics\\ University of Oregon\\ Eugene, OR
97403}

\address{Department of Mathematics\\ Wayne State University\\
Detroit, MI 48202}

\thanks{The first author was supported by NSF grant DMS0604354.  The second
author was supported by NSF grant DMS0803997.}

\email{ddugger@math.uoregon.edu}

\email{isaksen@math.wayne.edu}

\begin{abstract}
We present some data on the cohomology of the motivic Steenrod algebra
over an algebraically closed field of characteristic $0$.  Our
results are based on computer calculations and a motivic version of
the May spectral sequence.  We discuss features of
the associated Adams spectral sequence, and use these tools to give
new proofs of some results in classical algebraic topology.  
We also consider a motivic Adams-Novikov spectral sequence.  The
investigations reveal the existence of some stable motivic homotopy
classes that have no classical analogue.
\end{abstract}

\maketitle

\section{Introduction}

In modern algebraic geometry one studies varieties defined over
arbitrary fields, whereas classically this was done over only the real
or complex numbers.  The subject of motivic homotopy theory is an
attempt to generalize algebraic topology in the same way; rather than
study standard topological spaces, one studies a category of
``spaces'' defined over a fixed ground field.  The subject was put on
firm foundations by the paper \cite{MV}, and it encompasses several
areas of study like \'etale cohomology and algebraic $K$-theory.

One of the things that came out of \cite{MV} was the realization that
almost any object studied in classical algebraic topology could be
given a motivic analog.  In particular, one could define the motivic
stable homotopy groups of spheres.  The present paper begins an
investigation of the Adams spectral sequence, based on mod $2$ motivic
cohomology, that abuts to these groups.  Our results not only
contribute to the study of motivic phenomena, but we find that they
can be used to prove theorems about the {\it classical\/} stable
homotopy groups as well.

\medskip

In the motivic stable homotopy category (over a given field) there is
a bigraded family of spheres $S^{p,q}$.  It follows that one has a
bigraded family of stable homotopy groups $\pi_{p,q}(X)$ for any
motivic spectrum $X$, and consequently all generalized homology and
cohomology theories are bigraded.  For the sphere spectrum $S$, we
will abbreviate $\pi_{p,q}(S)$ to just $\pi_{p,q}$.

The object $S^{n,0}$ is like an ordinary sphere, and smashing with it
models the $n$-fold suspension in the triangulated structure.  The
object $S^{1,1}$ is the suspension spectrum of the scheme $\A^1-0$,
and represents a `geometric' circle.  The other spheres are obtained
by smashing together copies of $S^{1,1}$ and $S^{1,0}$, and taking
formal desuspensions.  One should take from this that the groups
$\pi_{p,0}$ are most like the classical stable homotopy groups,
whereas the groups $\pi_{p,q}$ for $q>0$ somehow contain more
geometric information.  We will see this idea again in
Remark~\ref{re:weight0}.  

The current knowledge about the groups $\pi_{p,q}$ is due to work of
Morel.  It follows from results in \cite{M3} that $\pi_{p,q}=0$ for
$p<q$.  In \cite{M2} it is shown that the groups $\pi_{n,n}$ 
can be completely described in terms of generators and relations using
the Milnor-Witt $K$-theory of the ground field.
This theorem of Morel is in some sense a motivic analog of the
classical computation of the stable $0$-stem.  

\medskip

Let $F$ be a field of characteristic zero, and let $\M_2$ denote the
bigraded motivic cohomology ring of $\Spec F$, with
$\Z/2$-coefficients.  Let $A$ be the mod $2$ motivic Steenrod algebra
over $F$.  All of these notions will be reviewed in
Section~\ref{se:one} below.  The motivic Adams spectral sequence is a tri-graded
spectral sequence with
\[ E_2^{s,t,u}=\Ext^{s,(t+s,u)}_A(\M_2,\M_2),
\]
and $d_r\colon E_r^{s,t,u} \ra E_r^{s+r,t-1,u}$.  Here $s$ is the
homological degree of the Ext group (the Adams filtration), and
$(t+s,u)$ is the internal bigrading coming from the bigrading on $A$
and $\M_2$: so $t+s$ is the topological dimension and $u$ is the
motivic weight.  The spectral sequence was first
studied in the paper \cite{M1}.  
It converges to something related
to the stable motivic homotopy group $\pi_{t,u}$; see
Corollary~\ref{co:conv} for a precise statement.

The motivic Steenrod algebra $A$ is very similar to the classical mod
$2$ Steenrod algebra, but the action of $A$ on $\M_2$ is nontrivial in
general. This extra feature seriously complicates the computation of
the Ext groups in the Adams $E_2$-term.  If the ground field $F$
contains a square root of $-1$, however, the action of $A$ on $\M_2$
is trivial and the Ext groups are much more accessible.  This is the
situation that we study in the present paper.  In fact, for added
convenience we almost entirely restrict to the case of algebraically
closed fields.

One goal of this paper is to present some computer calculations of the
$E_2$-term of the motivic Adams spectral sequence.  We have carried
out the computations far enough to discover several exotic elements
that have no classical analogues.

By the usual `Lefschetz principle', the Adams spectral sequence over
any algebraically closed field of characteristic zero takes the
same form as over the complex numbers.  One might as well restrict to
this case.  But over the complex numbers, there is a map from the
(tri-graded) motivic Adams spectral sequence to the (bi-graded)
classical one.  Comparing the two spectral sequences allows
information to pass in both directions.  On the one hand, we can
deduce motivic differentials from knowledge of the classical
differentials.  We can also turn this around, though, and use the
algebra of the motivic $E_2$-term to prove results about classical
differentials, and about extension problems as well.  This is
described further in Sections~\ref{se:top1} and ~\ref{se:top2} below.
We stress that these methods are purely algebraic; as soon as one
knows the motivic spectral sequence exists, and can do the related
algebraic calculations, some results about the {\it classical\/} Adams
spectral sequence come out almost for free.

It is useful to keep in mind the analogy of mixed Hodge theory and its
influence on topological gadgets like the Leray-Serre spectral
sequence.  Once it is known that the Leray-Serre differentials
preserve the Hodge weight, these differentials become easier to
analyze---there are fewer places where they can be nontrivial.  A
similar thing happens with the Adams spectral sequence and the motivic
weight.  The three-dimensional nature of the motivic Adams spectral
sequence allows the nontrivial groups to be more spread out, and also
allows the algebra of Massey products to work out in slightly
different ways than the classical story.  As a result, certain purely
topological phenomena become easier to analyze.  So far this technique
has only yielded results of casual interest, but a more thorough study
in higher dimensions might be fruitful.

\medskip

Throughout the article, we are considering only the situation where
the ground field is algebraically closed and has characteristic $0$.
The latter assumption is absolutely necessary at the moment, as the
motivic Steenrod algebra is unknown over fields of positive
characteristic.  The assumption that the ground field is algebraically
closed is less crucial.  It is very easy to extend from the
algebraically closed case to the case where the ground field contains
a square root of $-1$; see Remark~\ref{re:genF}.  When considering
fields which do not contain such a square root, however, the
difficulty of the calculations increases dramatically.  A future
paper \cite{DHI} will deal with this more general case.

\begin{remark}
In this paper we only deal with the motivic Adams spectral sequence
based on mod $2$ motivic cohomology.  Most aspects of our discussion
also work at odd primes, but the motivic Steenrod algebra in that case
is exactly isomorphic to the classical Steenrod algebra---the only
difference being the existence of the extra grading by weights.  The
$E_2$-term of the odd primary motivic Adams spectral sequence
therefore takes the same form as it does classically.
\end{remark}

\subsection{Organization of the paper}
In the first two sections we introduce the relevant background and set
up the cohomological Adams spectral sequence.  (It is probably better
to work with the {\it homological\/} Adams spectral sequence, but
working cohomologically allows us to postpone some motivic issues to
later in the paper.)  Then we discuss basic properties of the motivic
$\Ext$ groups and give the computation of the Adams $E_2$-term for $t
\leq 34$.  Next we show how to derive this computation by hand with
the motivic May spectral sequence.  Then we make some simple
computations of differentials, and use this to prove a theorem in
classical topology (Proposition~\ref{pr:h1hj}).

In Section~\ref{se:conv} we turn to convergence issues, and prove that
our Adams spectral sequence converges to the homotopy groups of the
$H$-nilpotent completion $S^{\wedge}_H$ of the sphere spectrum,
where $H$ is the motivic mod 2 Eilenberg-Mac Lane spectrum.
This entails dealing with the homological Adams spectral
sequence and related material.  Section~\ref{se:motAN} discusses the motivic
Adams-Novikov spectral sequence, first studied in \cite{HKO}, and uses
this to deduce further differentials in the motivic Adams spectral
sequence.  From this information we then prove a second theorem in
classical topology, namely Proposition~\ref{pr:kappaepsilon}.  

The appendices contain charts of the motivic Adams spectral sequence,
the motivic Adams-Novikov spectral sequence, and the $E_4$-term of the
motivic May spectral sequence.  These charts are best viewed in color.

\subsection{Notation and terminology}
The following notation is used in the paper:
\begin{enumerate}
\item
$\M_2$ is the mod 2 motivic cohomology of a point (i.e., the ground field).
\item
$A$ is the mod 2 motivic Steenrod algebra.
\item
$H^{*,*}(\blank)$ is mod 2 motivic cohomology.
\item
$H$ is the mod 2 motivic Eilenberg-Mac Lane spectrum.
\item
$\pi_{s,t}$ is the group of stable motivic homotopy classes of maps 
$S^{s,t} \map S^{0,0}$.
\item
$A_{\cl}$ is the classical mod 2 Steenrod algebra.
\item
$H^{*}(\blank)$ is classical mod 2 cohomology.
\item
$H_{\cl}$ is the classical mod 2 Eilenberg-Mac Lane spectrum.
\end{enumerate}

We are working implicitly with appropriate model categories of spectra
and of motivic spectra.  The categories of symmetric spectra \cite{HSS}
and of motivic symmetric spectra \cite{J} will work fine.

Several times in the paper we use without comment the fact that
the spheres are compact objects of the
motivic stable homotopy category.  That is, if $\Wedge_\alpha
E_\alpha$ is a wedge of motivic spectra, then $\oplus_\alpha
[S^{p,q},E_\alpha]\ra [S^{p,q},\Wedge_\alpha E_\alpha]$ is an
isomorphism.  In fact this is true not just for spheres, but for all
suspension spectra of smooth schemes; this is proven in \cite[Section 9]{DI2}.

\subsection{Acknowledgments}
This work was begun while the two authors were visitors at Stanford
University during the summer of 2007, and we are grateful to Gunnar
Carlsson for the invitation to be there.  The first author would also like
to acknowledge a very helpful visit to Harvard University during the
summer of 2008.  This visit was sponsored by Mike Hopkins and
supported by the DARPA FAthm grant FA9550-07-1-0555.

We are also grateful to Mark
Behrens, Robert Bruner, Mike Hill, and Michael Mandell for many
useful conversations.

\section{Background}
\label{se:one}

In this section we review the basic facts about motivic cohomology and motivic
Steenrod operations.

\subsection{The cohomology of a point}
Let $F$ be a field of characteristic $0$.  We
write $\M_2$ for the bigraded cohomology ring $H^{*,*}(\Spec F;\Z/2)$.
A theorem of Voevodsky \cite{V1} describes $\M_2$ explicitly in terms
of the Milnor $K$-theory of $F$.  Unfortunately this result doesn't
have a simple reference, and instead is a confluence of Corollary
6.9(2), Theorem 6.1, and Corollary 6.10 of \cite{V1}.  Unless
otherwise stated, we will always assume in this paper that $F$ is
algebraically closed.  For such fields the description of $\M_2$ has
a particularly simple
form:

\begin{thm}[Voevodsky]
The bigraded ring $\M_2$ is the polynomial ring $\F_2[\tau]$ on one
generator $\tau$ of bidegree $(0,1)$.
\end{thm}

In a bidegree $(p,q)$, we shall refer to $p$ as the topological degree 
and $q$ as the weight.

\subsection{The motivic Steenrod algebra}

We write $A$ for the ring of stable cohomology operations on mod $2$
motivic cohomology.  Voevodsky has computed this ring explicitly for
fields of characteristic zero: see \cite[Section 11]{V2} and
\cite[Theorem 1.4]{V3}.  In general, the ring is generated over $\M_2$
by the Steenrod operations $\Sq^i$.  Note that $\Sq^{2k}$ lies in
bidegree $(2k,k)$, whereas $\Sq^{2k-1}$ lies in bidegree $(2k-1,k-1)$.
These operations satisfy a complicated version of the Adem relations.
Again, the description simplifies quite a bit when the ground field is
algebraically closed:

\begin{thm}[Voevodsky]
The motivic Steenrod algebra $A$ is the $\M_2$-algebra generated by
elements $\Sq^{2k}$ and $\Sq^{2k-1}$ for all $k \geq 1$, of bidegrees
$(2k,k)$ and $(2k-1,k-1)$ respectively, and satisfying the following
relations for $a< 2b$:
\[ 
\Sq^a \Sq^b = 
\sum_{c} \binom{b-1-c}{a-2c} \tau^{?}\Sq^{a+b-c} \Sq^c.
\]
\end{thm}

Note the coefficients $\tau^?$ in the Adem relation above.  Here the
``?'' denotes an exponent which is either $0$ or $1$.  We could
explicitly write the exponent, but only at the expense of making the
formula appear more unwieldy.  The exponent is easily determined in
practice, because it is precisely what is needed in order to make the
formula homogeneous in the bidegree.  

For example, consider the formula $\Sq^2 \Sq^2 = \tau^? \Sq^3 \Sq^1$.
Since $\Sq^2$ has bidegree $(2,1)$, the left side has total bidegree
$(4,2)$.  On the other hand, $\Sq^1$ has bidegree $(1,0)$ and $\Sq^3$
has bidegree $(3,1)$, so we require one $\tau$ on the right side in
order to make the formula homogeneous.  In other words, we have the
motivic Adem relation $\Sq^2 \Sq^2 = \tau \Sq^3 \Sq^1$.  It turns out
this is representative of what happens in all the motivic Adem
relations: the $\tau$ appears precisely when $a$ and $b$ are even and $c$
is odd.

The Steenrod operations act on the motivic cohomology $H^{*,*}(X)$ of any
smooth scheme $X$.  In particular, they act on the cohomology of
$\Spec F$.  In our case, where $F$ is algebraically closed, $\M_2$ is
concentrated entirely in topological degree $0$.  It follows that the
Steenrod operations (other than the identity) act trivially on $\M_2$ for
dimension reasons.

Just as in the classical situation,
the admissible monomials form a basis for $A$
as an $\M_2$-module \cite{S} \cite{MT}.

\subsection{The Milnor basis}
\label{se:milnor}

A theorem of Voevodsky gives an explicit description of the dual of $A$
\cite{V2}.  As an algebra $A_*$ can be described as
the subalgebra $\M_2[\zeta_i, \tau^{-1} \zeta_i^2]_{i\geq 1}\subseteq
\M_2[\tau^{-1},\zeta_1,\zeta_2,\ldots]$.  Here $\tau$ has (homological) bidegree $(0,-1)$
and $\zeta_i$ has bidegree $(2^i-1,2^{i-1}-1)$.

The elements $\tau^{-1} \zeta_i^2$
are indecomposable in $A$, but their names are a useful way of encoding the 
relation $(\zeta_i)^2 = \tau \cdot \tau^{-1} \zeta_i^2$.
In the notation of \cite{V2}, $\zeta_i$ corresponds to $\tau_i$
and $\tau^{-1} \zeta_i^2$ corresponds to $\xi_{i+1}$.
Under the canonical map $A_*\ra
A_*^{\cl}$ to the classical dual Steenrod algebra---see
Section~\ref{subsctn:compare-algebra} below---the element $\zeta_i$
maps to $\xi_i$ and $\tau^{-1}\zeta_i^2$ maps to $\xi_i^2$.

The comultiplication on $A_*$ is identical to the comultiplication
on the classical Steenrod algebra \cite{Mi}, except that appropriate powers of
$\tau$ must be inserted to make the formulas homogeneous in the bidegree.
Namely, the coproduct on $\zeta_k$ takes the form
$\sum_i \tau^? \zeta^{2^i}_{k-i} \otimes \zeta_i$.

There is an evident  basis for $A_*$ as an $\M_2$-module
consisting of monomials $\tau^?\zeta_1^{r_1}\zeta_2^{r_2}\cdots$.
Here $?$ is the smallest power of $\tau$ that gives an expression
lying in $A_*$, namely $-\sum_i \lfloor \frac{r_i}{2}\rfloor$.  
Just as in the classical situation, this basis for $A_*$ yields a dual
basis for $A$ (as an $\M_2$-module) called the Milnor basis.
The Milnor basis consists of elements of the form
$P^R$, where $R = (r_1,r_2, \ldots)$ ranges over all finite sequences of 
non-negative integers.  
Here $P^R$ is dual to $\tau^? \zeta_1^{r_1} \zeta_2^{r_2} \ldots$.
The bidegree of $P^R$ is easily calculated, and equals
\[
\left( \sum r_i (2^i-1), 
\sum \left\lfloor \frac{r_i(2^i - 1)}{2} \right\rfloor \right).
\]

Products $P^R P^S$ can be computed with matrices and multinomial coefficients
just as in \cite{Mi}, except that some terms require a power of $\tau$
as a coefficient.  Similarly to the motivic Adem relations, 
these coefficients are easy to calculate; they are exactly
what are needed in order to make the formulas homogeneous in the bidegree.
However, unlike the Adem relations, it is sometimes necessary to use
exponents greater than one.  (The first occurrence of $\tau^2$ in these
formulas occurs in topological dimension $26$.)

We make use of the Milnor basis in Section \ref{se:May} when
we discuss the motivic May spectral sequence.  Also,
some of our computer calculations use this basis.

\subsection{Comparison with the classical Steenrod algebra}
\label{subsctn:compare-algebra}

There is a topological realization functor from
motivic spectra over $\C$ to
ordinary spectra, as described in \cite{MV} \cite{DI1}.  This functor
is uniquely determined (up to homotopy) by the fact that it
preserves homotopy colimits and weak equivalences, and sends the
motivic suspension spectrum of a smooth scheme $X$ to the ordinary
suspension spectrum of its topological space of complex-valued points
$X(\C)$.  We generalize this notation, so that if $E$ is
any motivic spectrum then $E(\C)$ denotes its topological realization.

Let $H$ be the mod 2 motivic Eilenberg-Mac Lane spectrum, i.e., the
motivic spectrum that represents mod 2 motivic cohomology.  
It follows from \cite[Lemma 4.36]{V3} that $H(\C)$ is the
classical mod 2 Eilenberg-Mac Lane spectrum $H_{\cl}$.  
So topological realization induces natural transformations
$H^{p,q}(X) \map H^p(X(\C))$ denoted $\alpha \mapsto \alpha(\C)$.  We
implicitly consider $H^*(X(\C))$ to be a bigraded object concentrated
in weight $0$.

\begin{defn}
For any motivic spectrum $X$, let
\[
\theta_X\colon H^{*,*}(X) \otimes_{\M_2} \M_2[\tau^{-1}] \map 
H^*(X(\C)) \otimes_{\F_2} \M_2[\tau^{-1}]
\]
be the $\M_2[\tau^{-1}]$-linear map
that takes a class $\alpha$ of weight $w$ in $H^{*,*}(X)$
to $\tau^w \alpha(\C)$.  
\end{defn}

In order for $\theta_X$ to be well-defined, we must observe that
$\tau(\C)$ equals 1 in the singular cohomology of a point.
This follows by splitting the topological realization map into two
pieces
\[ H^{p,q}(\Spec \C) \ra H^p_{et}(\Spec \C;\mu_2^{\otimes q}) \ra 
H^p_{sing}(\Spec \C)
\]
where the group in the middle is \'etale cohomology.  For $p\leq q$
the first map is an isomorphism by \cite{V1}, and it is a classical
result of Grothendieck and his collaborators \cite[XI]{SGA4} that the second map is an
isomorphism.  In particular, this holds when $p=0$ and $q=1$.  Since
the group in question is $\Z/2$, it must be that $\tau$ maps to $1$.

\begin{lemma}
\label{lemma:theta}
The map $\theta_X$ is an isomorphism of bigraded $\M_2[\tau^{-1}]$-modules
if $X$ is the motivic sphere spectrum
or if $X$ is the motivic Eilenberg-Mac Lane spectrum.
\end{lemma}

In the following proof, 
we write $A_{\cl}$ for the classical topological mod $2$ Steenrod algebra.

\begin{proof}
This follows directly from the description
of the motivic cohomology of a point and of the motivic Steenrod algebra.

For the sphere spectrum, $\theta_X$ is just the map
$\M_2 \otimes_{\M_2} \M_2[\tau^{-1}] \map \F_2 \otimes_{\F_2} \M_2[\tau^{-1}]$,
which is obviously an isomorphism.

For the motivic Eilenberg-Mac Lane spectrum, note that 
$H^{*,*}(X)$ is the motivic Steenrod algebra, and $H^*(X(\C))$
is the classical topological Steenrod algebra.  The map $\theta_X$
takes $\Sq^{2k}$ to $\tau^{-k} \Sq^{2k}$, and takes
$\Sq^{2k+1}$ to $\tau^{-k} \Sq^{2k+1}$.  But we have that
$A \otimes_{\M_2} \M_2[\tau^{-1}]$ is free as an $\M_2[\tau^{-1}]$-module
on the admissible monomials.  
Also, 
$A_{\cl} \otimes_{\F_2} \M_2[\tau^{-1}]$ is free as an 
$\M_2[\tau^{-1}]$-module on the admissible monomials.  It follows that
$\theta_X$ is an isomorphism.
\end{proof}

\begin{cor}
\label{cor:theta}
The map
$A \otimes_{\M_2} \M_2[\tau^{-1}] \map A_{\cl} \otimes_{\F_2} \M_2[\tau^{-1}]$
that takes $\Sq^{2k}$ to $\tau^{-k} \Sq^{2k}$ and takes
$\Sq^{2k-1}$ to $\tau^{-k} \Sq^{2k-1}$ is an isomorphism of bigraded rings.
\end{cor}

\begin{proof}
This follows immediately from Lemma \ref{lemma:theta}.
\end{proof}

Before giving the final lemmas of this section we need a brief
discussion of a finite type condition.  For the remainder of this
subsection we work over a general field $F$ (not
necessarily algebraically closed).   

\begin{defn}
\label{defn:Chow-weight}
Given an element $x$ of bidegree $(p,q)$ in a bigraded group,
the \dfn{Chow weight} $C(x)$ of $x$ is the integer $2q-p$.
\end{defn}

The terminology is motivated
by the fact that $2q-p$ is a natural index in the 
higher Chow group perspective on motivic cohomology \cite{Bl}.
It is known that the mod $2$
motivic cohomology groups $H^{p,q}(X)$ of any smooth scheme are
concentrated in non-negative Chow weight and also in the range $p \geq 0$ (see
\cite[Thm. 19.3]{MVW} and \cite[Thm. 7.8]{V1}).

\begin{defn}
Let $\{n_\alpha=(p_\alpha,q_\alpha)\}_{\alpha\in S}$ be a set of
bidegrees.  This set is \dfn{motivically finite type} if for 
any $\alpha$ in $S$, there are only finitely many $\beta$ in $S$
such that $p_\alpha \geq p_\beta$ and 
$2q_\alpha - p_\alpha \geq 2q_\beta - p_\beta$.
\end{defn}

For example, this condition is satisfied if:
\begin{enumerate}[(i)]
\item
the $p_\alpha$'s are bounded below, and
\item 
for each $\alpha$, there are only finitely many $\beta$'s for which
$p_\beta=p_\alpha$ and $q_\beta\leq q_\alpha$.  
\end{enumerate}

%
%

\begin{defn}
Let $X$ be a motivic spectrum.
A wedge of suspensions
$\Wedge_{\alpha\in S}
\Sigma^{n_\alpha}X$ is
\dfn{motivically finite type} if the index set of bidegrees is 
motivically finite type.
Likewise, a free graded module over $\M_2$ is 
\dfn{motivically finite type} if it
is of the form $\oplus_{\alpha\in S}\Sigma^{n_\alpha}\M_2$ for a
motivically finite type index set of bidegrees.
\end{defn}

\begin{lemma}
\label{le:sum=prod}
For any smooth scheme $X$ and any motivically finite type index set of bidegrees
$\{n_\alpha\}_{\alpha \in S}$, the canonical map
\[ \bigoplus_\alpha \Sigma^{n_\alpha} H^{*,*}(X) \ra \prod_\alpha
\Sigma^{n_\alpha} H^{*,*}(X) 
\]
is an isomorphism.  Here $\Sigma^{n_\alpha}$ denotes the algebraic
shifting of bigraded modules.  
\end{lemma}

\begin{proof}
The motivically finite type condition guarantees that in each
bidegree, only finitely many terms of the product are nonzero.  
\end{proof}

\begin{lemma}
\label{lemma:triangulated-theta}
Let $X$ belong to the smallest triangulated category of motivic spectra
that contains the spheres and motivically finite type
 wedges of suspensions of the mod 2
motivic Eilenberg-Mac\,Lane spectrum $H$.  Then
$\theta_X$ is an isomorphism.
\end{lemma}

\begin{proof}
Let $\cC$ be the full subcategory of the motivic stable homotopy category
consisting of all
motivic spectra $A$ such that $\theta_A$ is an isomorphism.  It is
clear that $\cC$ is closed under suspensions and cofiber sequences,
and by Lemma~\ref{lemma:theta} $\cC$ contains the sphere spectrum and $H$.
It only
remains to show that $\cC$ contains all motivically finite type wedges of $H$.  

Let $W=\Wedge_{\alpha} \Sigma^{n_\alpha}H$ be a motivically finite type wedge of
suspensions of $H$.  Then $H^{*,*}(W)=\prod_{\alpha} \Sigma^{n_\alpha}H^{*,*}(H)$, and
by Lemma~\ref{le:sum=prod} this equals
$\oplus_{\alpha} \Sigma^{n_\alpha}H^{*,*}(H)$ as modules over $\M_2$.  We
may then distribute the direct sum over the tensors that appear in the
definition of $\theta_W$, and conclude that $\theta_W$ is an
isomorphism based on the fact that $\theta_H$ is.  
\end{proof}


\section{The motivic Adams spectral sequence}
\label{se:motadams}

Starting with the motivic sphere spectrum $S^{0,0}$,
we inductively construct an Adams resolution
\[ \xymatrix{
 & K_2 & K_1 & K_0\\
\cdots \ar[r] & X_2\ar[r]\ar[u] & X_1 \ar[r]\ar[u] & X_0 \ar[u]
\ar@{=}[r] & S^{0,0}
}
\]
where each $K_i$ is a motivically finite type
wedge of suspensions of $H$, $X_i \ra K_i$ is
surjective on mod 2 motivic cohomology, and $X_{i+1}$ is the homotopy
fiber of $X_i \ra K_i$.
Applying $\pi_{*,u}$ gives us an exact couple---one for each $u$---and
therefore a $\Z$-graded family of spectral sequences indexed by $u$.
The usual arguments show that the $E_2$-term is
$\Ext_A^{s,(t+s,u)}(\M_2,\M_2)$ and that the spectral sequence abuts to
something related to the stable motivic homotopy group $\pi_{t,u}$.  

\begin{remark}
In Section~\ref{se:conv} we will prove that the spectral sequence
converges to the homotopy groups of the $H$-nilpotent completion of
the sphere spectrum $S$, denoted $S^{\wedge}_H$.  
We will also discuss the $H$-homology
version of this tower, in which each $K_i$ is replaced by $X_i\Smash
H$.  Both of these discussions involve some technicalities that we
wish to avoid for the moment.
\end{remark}

\subsection{Comparison with the classical Adams spectral sequence}

Applying the
topological realization functor to our Adams resolution, we obtain a tower 
\[ \xymatrix{
 & K_2(\C) & K_1(\C) & K_0(\C)\\
\cdots \ar[r] & X_2(\C)\ar[r]\ar[u] & X_1(\C) \ar[r]\ar[u] & X_0(\C) \ar[u]
\ar@{=}[r] & S 
}
\]
of homotopy fiber sequences of ordinary spectra.
It is not a priori clear that this is a classical
Adams resolution, however. This requires an argument.

\begin{prop}
Let $f: X_i\ra K_i$ be one of the maps in the motivic Adams resolution.
The map $f(\C): X(\C) \ra K(\C)$ is surjective on mod $2$ singular cohomology.
\end{prop}

\begin{proof}
We know that the map
$H^{*,*}(K_i) \map H^{*,*}(X_i)$ is surjective, so the map
$H^{*,*}(K_i) \otimes_{\M_2} \M_2[\tau^{-1}] \map 
H^{*,*}(X_i) \otimes_{\M_2} \M_2[\tau^{-1}]$ is also surjective.
It follows that 
$ H^*(K_i(\C)) \otimes_{\F_2} \M_2[\tau^{-1}] \map
H^*(X_i(\C)) \otimes_{\F_2} \M_2[\tau^{-1}]$ is surjective,
since $\theta_{K_i}$ and $\theta_{X_i}$ are isomorphisms
by Lemma \ref{lemma:triangulated-theta}.
Restriction to weight zero shows that the map
$H^*(K_i(\C)) \map H^*(X_i(\C))$ is surjective.
\end{proof}

Topological realization gives natural maps $\pi_{p,q}(Z)\ra
\pi_p(Z(\C))$ for any motivic spectrum $Z$, so that we get a map from
the homotopy exact couple of $(X,K)$ to that of $(X(\C),K(\C))$.  In
this way we obtain a map of spectral sequences from the motivic Adams
spectral sequence for the motivic sphere spectrum 
to the classical Adams spectral sequence for the classical 
sphere spectrum.

\subsection{Comparison of $\Ext$ groups}
\label{subsctn:ext-compare}

Our next task is to compare $\Ext$ groups 
over the motivic Steenrod algebra $A$ to
$\Ext$ groups over the classical Steenrod algebra $A_{\cl}$.  

Note that $\Ext^{0,*}_A(\M_2,\M_2)=\Hom^*_A(\M_2,\M_2)=\F_2[\tilde{\tau}]$
where $\tilde{\tau}$ is the dual of $\tau$ and has bidegree $(0,-1)$.
We will abuse notation and often just write $\tau$ instead of
$\tilde{\tau}$, but we will consistently write $\tilde{\M}_2$ for
$\F_2[\tilde{\tau}]$.

For fixed $s$ and $t$,
$\Ext_A^{s,(t+s,*)}(\M_2,\M_2)$ is an $\tilde{\M}_2$-module; it
therefore decomposes as a sum of free modules and
modules of the form $\tilde{\M}_2/\tau^k$.  
The following result shows that the free
part coincides with $\Ext$ over the classical Steenrod algebra.

\begin{prop} 
\label{prop:compare-ext}
There is an isomorphism of rings
\[\Ext_A(\M_2,\M_2)\tens_{\tilde{\M}_2}\tilde{\M}_2[\tau^{-1}] \iso
\Ext_{A_{\cl}}(\F_2,\F_2)\tens_{\F_2} \F_2[\tau,\tau^{-1}].
\]
\end{prop}

\begin{proof}
Throughout the proof, we use that $\M_2 \map \M_2[\tau^{-1}]$ and 
$\F_2 \map \F_2[\tau,\tau^{-1}]$ are flat, so that the associated tensor product functors are exact.

There is a natural isomorphism
\[
\Hom_{A_{\cl}}(\F_2, \blank) \otimes_{\F_2} \F_2[\tau,\tau^{-1}] \map 
\Hom_{A_{\cl} \otimes_{\F_2} \F_2[\tau,\tau^{-1}]} 
(\F_2[\tau,\tau^{-1}], (\blank) \otimes_{\F_2} \F_2[\tau,\tau^{-1}]),
\]
so the induced map on derived functors is also an isomorphism.  In particular,
we obtain an isomorphism
\[
\Ext_{A_{\cl}}(\F_2, \F_2) \otimes_{\F_2} \F_2[\tau,\tau^{-1}] \map 
\Ext_{A_{\cl} \otimes_{\F_2} \F_2[\tau,\tau^{-1}]}
(\F_2[\tau,\tau^{-1}], \F_2[\tau,\tau^{-1}]).
\]

Recall from Corollary \ref{cor:theta}
that $A[\tau^{-1}]$ is isomorphic to
$A_{\cl} \otimes_{\F_2} \F_2[\tau,\tau^{-1}]$, so it remains to
show that there is an isomorphism
\[
\Ext_A (\M_2, \M_2) \otimes_{\tilde{\M}_2} \tilde{\M}_2[\tau^{-1}] \map
\Ext_{A[\tau^{-1}]}(\M_2[\tau^{-1}],\M_2[\tau^{-1}]).
\]
As in the previous paragraph, this follows from the natural isomorphism
\[
\Hom_A(\M_2, \blank) \otimes_{\tilde{\M}_2} \tilde{\M}_2[\tau^{-1}] \map 
\Hom_{A[\tau^{-1}]}(\M_2[\tau^{-1}], (\blank) \otimes_{\M_2} \M_2[\tau^{-1}] ).
\]
\end{proof}


\section{Computation of $\Ext_A(\M_2,\M_2)$}

In this section we are
concerned with computing the groups
$\Ext_A^{s,(t+s,u)}(\M_2,\M_2)$, which we abbreviate to $\Ext^{s,(t+s,u)}$.
Here $s$ is the homological degree, and
$(t+s,u)$ is the internal bidegree.  Recall that $t+s$ corresponds to the
usual topological grading in the classical Steenrod algebra, and $u$
is the motivic weight.

A chart showing the motivic $\Ext$ groups is given in
Appendix~\ref{se:motivicext}.  There are several things about the
chart we must explain, however.  First, the reader should view these
pictures as a projection of a three-dimensional grid, where the $u$
axis comes out of the page.  Although our diagrams suppress this third
direction, it is always important to keep track of the weights.  Our
two-dimensional charts are drawn using the classical conventions, so
that $t$ is on the horizontal axis and $s$ is on the vertical axis.  We
refer to $t$ as the \emph{stem} and $s$ as the \emph{Adams
filtration}.  Note that column $t$ of the chart contains groups
related to $\pi_{t,*}$.

Recall from Section \ref{subsctn:ext-compare} that 
$\Ext^{0,(*,*)}$ is the algebra $\tilde{\M}_2$,
i.e., a polynomial algebra over $\F_2$ with one generator $\tau$ in
bidegree $(0,-1)$.  The Yoneda product makes $\Ext^{*,(*,*)}$ into
a graded module over $\Ext^{0,(*,*)}$.  Therefore, $\Ext^{*,(*,*)}$
decomposes into a direct sum of copies of $\tilde{\M}_2$ and 
$\tilde{\M}_2/ \tau^k$.  It turns out that in the range $t-s\leq 34$,
the only value of $k$ for which the latter occurs is $k=1$.  

In our chart,
each black circle represents a copy of $\tilde{\M}_2$.  The small numbers
in the chart indicate the weights of the generators of some of these copies
(the weights of the rest of the generators are implied by the multiplicative
structure).

Each red circle represents a copy of $\tilde{\M}_2 / \tau$.  Note that
this submodule is concentrated in a single weight, which is either
given or implied.  By Proposition \ref{prop:compare-ext}, if we delete
the red circles from our chart, the black circles correspond precisely
to the classical $\Ext$ chart \cite[Appendix 3]{R}.

The results in this section were obtained by computer calculations.
The software computed a resolution for $\M_2$ over $A$, and from this
information extracted $\Ext$ groups and product information in the
usual way.  In the range $t-s \leq 34$, the computations have been checked
by two independent software packages, and have been further
corroborated by work with the May spectral sequence described in
Section~\ref{se:May}.

The software computed a minimal free resolution over the graded ring
$A$.  In the classical case one computes $\Ext$ by applying
$\Hom_A(\blank,\F_2)$, which makes all the differentials vanish.
In the motivic case we are applying $\Hom_A(\blank,\M_2)$, which
leaves powers of $\tau$ inside the differentials.  The software then
computed the cohomology of the resulting complex, regarded as a
complex of modules over $\M_2$.

\subsection{Product structure}

Multiplication by $h_0$ is represented by vertical black lines (and
some blue lines, to be described below).  The black arrow in the
$0$-stem indicates an infinite tower of elements of the form $h_0^k$.

Multiplication by $h_1$ is represented by diagonal black or red lines.
The diagonal red arrows indicate an infinite sequence of copies of
$\tilde{\M}_2 / \tau$ which are connected by $h_1$ multiplications.
In other words, the red arrows represent a family of elements of the
form $h_1^k x$ such that $\tau h_1^k x$ equals zero.
In particular, note that $h_1$ is not nilpotent, although $\tau h_1^4$ is
zero.  This is the first example of non-classical phenomena in the
motivic $\Ext$ groups.

Multiplication by $h_2$ is represented by green lines of slope $1/3$
(and some magenta lines described below).

Inspection of the 3-stem leads us to the next non-classical phenomenon.
The element $h_1^3$ has weight 3, while the element
$h_0^2 h_2$ has weight 2.  Classically these elements are equal, but this
cannot happen here.  Instead we have the relation
\[
h_0^2 h_2 = \tau h_1^3.
\]
Both sides of this equation have weight 2.

The blue vertical lines indicate a multiplication by $h_0$ whose value
is $\tau$ times a generator.  Similarly, a magenta line of slope $1/3$
indicates a multiplication by $h_2$ whose value is 
$\tau$ times a generator.
There are many occurrences of this phenomenon.  One
example is the relation
\[
h_0^2 Ph_2 = \tau h_1^2 Ph_1,
\]
which is implied by Massey products and the relation $h_0^2 h_2 = \tau h_1^3$.
We will discuss this later.
Other interesting examples include
\[
h_0 f_0 = \tau h_1 e_0, 
\hspace{10pt} h_0^2 j = \tau h_1 Pe_0, 
\hspace{10pt} h_0^2 k = \tau h_1 d_0^2, 
\hspace{10pt} h_0^5 r = \tau c_0 Pd_0.
\]

Note the element labelled $[\tau g]$ in the 20-stem with weight 11.
We emphasize that this element is indecomposable, i.e., it is not
divisible by $\tau$ (and the brackets are there to help us remember
this).  The choice of name makes other formulas work out more
consistently, as we shall see later (see also
Remark~\ref{re:why-g}). Similarly, note the indecomposable elements
$[h_2 g]$ and $[h_3 g]$.  We have $\tau[h_2g]=h_2[\tau g]$ and
similarly for $[h_3g]$, although in the latter case both $\tau[h_3g]$
and $h_3[\tau g]$ are zero.

We also observe that $c_0 d_0$ and $c_0 e_0$ are non-zero, even
though their classical versions are zero.  However, they are killed by
$\tau$, which is consistent with Proposition \ref{prop:compare-ext}
(see also Remark \ref{re:h1-stable}).

Finally, note that $h_0^5 r$ equals $\tau c_0 Pd_0$.  The element
$c_0 Pd_0$ supports multiplications by $h_1$, even though $\tau c_0 Pd_0$
cannot since it is divisible by $h_0$.

\begin{remark}
A look at the $\Ext$ chart in Appendix~\ref{se:motivicext} shows a
connection between the new $h_1$-towers (in red) and the spots where
there is a shifted $h_0$ (in blue)---that is, where multiplying $h_0$
by a generator gives $\tau$ times another generator.  This makes a
certain amount of sense: classically, $h_1^4$ is zero {\it because\/}
$h_1^3$ is a multiple of $h_0$, and similarly for $h_1^3Ph_1$,
$h_1^4d_0$, $h_1^2e_0$, etc.  In the motivic world $h_1^3$ is not a
multiple of $h_0$, and so there is no reason for $h_1^4$ to
vanish.  This connection breaks down, though, in the case of
$h_1^4h_4$.  Even though $h_1^3h_4$ is not a multiple of $h_0$ in the
motivic world, $h_1^4h_4$ still vanishes.  The reason for
this is tied to the class $[\tau g]$ in $\Ext^{4,(24,11)}$, and the
fact that it lives in weight $11$ instead of weight $12$---i.e., that
our generator is $[\tau g]$ and not $g$.  This can be explained
rigorously in terms of the May spectral sequence; see
Remark~\ref{re:why-g}.
\end{remark}

\begin{remark}
\label{re:weight0}
In the range $0\leq t-s\leq 34$, the weight $0$
piece of the motivic $E_2$-term coincides exactly with the classical
Adams $E_2$-term.  That is, all the $\tau$-torsion has disappeared by
the time one reaches weight $0$.  This is a reflection of the
principle, mentioned in the introduction, that the stable homotopy
groups $\pi_{n,0}$ are most like the classical stable homotopy groups.
\end{remark}

\subsection{Massey products}
In this section, all Massey products have zero indeterminacy unless
explicitly stated otherwise.

The symbol $P$ is used for the usual periodicity operator 
$\langle h_3, h_0^4, \blank \rangle$, defined on any class that is
killed by $h_0^4$.  Note that $P$ increases weights by $4$.

We can now understand the relation $h_0^2 Ph_2 = \tau h_1^2 P h_1$.
Starting with $\langle h_3, h_0^4, h_2 \rangle h_0^2$,
we shuffle to obtain $\langle h_3, h_0^4, h_0^2 h_2 \rangle$,
which equals $\langle h_3, h_0^4, \tau h_1^3 \rangle$.  By shuffling
again, this simplifies to 
$\langle h_3, h_0^4, h_1 \rangle \tau h_1^2$, which is $\tau h_1^2 Ph_1$.

The comparison between the motivic $\Ext$ groups and the classical
$\Ext$ groups is compatible with Massey products.  This fact can be
used to compute motivic Massey products based on known classical
computations.  
For example, consider the Massey product $\langle h_2^2, h_0, h_1 \rangle$.  
Classically, this Massey product equals $c_0$.  Therefore, the motivic
Massey product must be non-zero.  The Massey product also has weight 5.  
An inspection of
the chart shows that there is just one possible value for this Massey product,
which we have labelled $c_0$.

The method of the previous paragraph allows us to compute the following
motivic Massey products, based on known classical formulas: 
\begin{align*}
c_0 &= \langle h_1,h_0,h_2^2\rangle \\
d_0 &= \langle h_0,h_2^2,h_0,h_2^2\rangle \\
e_0 &= \langle h_1,h_2,c_0,h_2\rangle \\
f_0 &= \langle h_0^2,h_3^2,h_2 \rangle \\
[\tau g] &= \langle h_2,h_1,h_0,h_0h_3^2\rangle \\
k &= \langle h_0^2,h_3^2,d_0\rangle \\
r &= \langle h_0^2,h_3^2,h_3^2,h_0^2\rangle.
\end{align*}
Note that the formula for $f_0$ has indeterminacy $\{ 0, \tau h_1^3 h_4\}$,
just as in the classical case.  
These formulas are printed in \cite{T2}, which gives citations
to original sources.

Something slightly different occurs with the Massey product
$\langle h_0^2,h_3^2,h_1,h_0\rangle$.  The weight here is 9.
Classically, this Massey product equals $e_0$, but the weight of 
motivic $e_0$ is 10.  Therefore, we deduce the motivic formula
\[
\tau e_0 = \langle h_0^2,h_3^2,h_1,h_0\rangle.
\]

Another difference occurs with the Massey products 
$\langle h_0, c_0, d_0 \rangle$ and $\langle h_0, c_0, e_0 \rangle$,
which are classically equal to $i$ and $j$.
Motivically, these Massey products are not defined because
$c_0 d_0$ and $c_0 e_0$ are non-zero.  However, $\tau c_0 d_0$
and $\tau c_0 e_0$ are zero, so the Massey products
$\langle h_0, c_0, \tau d_0 \rangle$ and $\langle h_0, c_0, \tau e_0 \rangle$
are defined.  By comparison to the classical situation again, we deduce
the motivic formulas
\[
i = \langle h_0, c_0, \tau d_0 \rangle, \hspace{10pt}
j = \langle h_0, c_0, \tau e_0 \rangle.
\]
We also have the similar formula $k = \langle h_0, c_0, [\tau g] \rangle$,
but this is slightly different because $[\tau g]$ is not decomposable.

Note that $Pg$ is not defined, since $g$ does not exist.  However,
$P ([\tau g])$ equals $\tau d_0^2$.  We might abuse notation and write
$[Pg]$ as another name for $d_0^2$.

Consider the element $[h_2 g]$ in the 23-stem.  Recall that this element
is indecomposable, but $\tau [h_2 g]=h_2[\tau g]$.
Computer calculations tell us that
\[
[h_2 g] = \langle h_1, h_1^3 h_4, h_2\rangle.
\]
Note that the indeterminacy is zero,
since $[h_2 g]$ is not divisible by $h_2$.
On the other hand, the classical Massey product
$\langle h_1, h_1^3 h_4, h_2\rangle$ has indeterminacy $\{ 0, h_2 g\}$.  

Similarly, we have the formula
\[
[h_2 g] = \langle h_2, h_1^4, h_4 \rangle,
\]
with zero indeterminacy again.  The corresponding classical Massey product
is $\langle h_2, 0, h_4 \rangle$, with indeterminacy $\{ 0, h_2 g \}$.

Computer calculations have verified that the indecomposable $[h_3 g]$
equals $\langle h_3, h_1^4, h_4 \rangle$ and 
$\langle h_1, h_1^3 h_4, h_3 \rangle$.  These formulas have no classical
analogues, since $h_3 g$ equals zero classically.


\section{The motivic May spectral sequence}
\label{se:May}

In this section we explain how one can compute the groups
$\Ext_A^{s,(t+s,u)}(\M_2,\M_2)$ without the aid of a computer.  These
groups can be produced by a variation on the classical May
spectral sequence \cite{Ma}. 

We continue to assume that the ground field $F$ is algebraically
closed, so that $\M_2=\F_2[\tau]$.  Let $I$ be the two-sided ideal of
the motivic Steenrod algebra generated by the $\Sq^i$'s; equivalently,
$I$ is the kernel of algebra projection map $A\ra \M_2$.  

Let $\Gr_I(A)$ denote the associated graded algebra $A/I \oplus I/I^2
\oplus I^2/I^3 \oplus\cdots$.  Note that this is a tri-graded algebra,
with two gradings coming from $A$ and one from the $I$-adic valuation.
We refer to the $I$-adic valuation of an element as its May filtration.
When indexing elements, we will always write the May filtration first.

The motivic May spectral sequence has the form
\[ E_2=\Ext^{s,(a,b,c)}_{\Gr_I(A)}(\M_2,M_2) \Rightarrow
\Ext_A^{s,(b,c)}(\M_2,\M_2).\]
As usual, it can be obtained by filtering the cobar complex by powers
of $I$.  

Let $A_{\cl}$ denote the classical Steenrod algebra over $\F_2$, and
let $I_{\cl}$ be the ideal generated by the $\Sq^i$'s.  
We will show that $\Gr_I(A)$ is easily described in terms of 
$\Gr_{I_{\cl}} (A_{\cl})$.  The $E_2$-terms of the motivic 
and classical May spectral
sequence are very similar, but the differentials are different.

We recall some notation and ideas from \cite{Ma}.  Given a sequence
$R = (r_1, r_2, \ldots)$, let $r_i = \sum_k a_{ik} \cdot 2^k$ be the
2-adic expansion of $r_i$.  Define $v(R)$ to be the integer 
$\sum_{i,k} i a_{ik}$.

The relations in the motivic Steenrod algebra with respect to the Milnor
basis are of the form
$P^R P^S = \sum_X \tau^{u(X)} b(X) P^{T(X)}$.
Here $X$ ranges over certain matrices as described in Theorem 4b of \cite{Mi};
$u(X)$ is a non-negative integer determined by the weights of $P^R$, $P^S$,
and $P^{T(X)}$; and $b(X)$ is a multinomial coefficient depending on $X$.

Recall the Chow weight from Definition \ref{defn:Chow-weight}.
For any Milnor basis element $P^R$, observe that the Chow weight $C(P^R)$ is
equal to the negative of the number of odd integers occurring in $R$; 
this follows
immediately from the formula for bidegrees given in Section \ref{se:milnor}.
Note that $u(X)$ is equal to $C(P^R) + C(P^S) - C(P^{T(X)})$.

\begin{lemma}
\label{lem:u(X)}
Let $P^R P^S = \sum_X \tau^{u(X)} b(X) P^{T(X)}$ be a relation
in the motivic Steenrod algebra with respect to the Milnor basis.
If $b(X) = 1$ and $v(T(X)) = v(R) + v(S)$, then $u(X)$ equals zero.
\end{lemma}

\begin{proof}
Let $R = (r_1, \ldots)$ and $S = (s_1, \ldots)$.
Let $r_i = \sum_k a_{ik} \cdot 2^k$ and $s_j = \sum_k b_{jk} \cdot 2^k$
be the $2$-adic expansions of $r_i$ and $s_j$.
Then $C(P^R)$ is equal to $-\sum_i a_{i0}$,
and $C(P^S)$ is equal to $-\sum_j b_{j0}$.

Let $X = (x_{ij})$, and let $x_{ij} = \sum_k e_{ijk} \cdot 2^k$ be
the $2$-adic expansion of $x_{ij}$.  Assume that $b(X) = 1$ and that
$v(T(X))$ equals $v(R) + v(S)$.
By Lemma 2.3 of \cite{Ma}, it follows that
$a_{i0} = e_{i00}$ and $b_{j0} = \sum_i e_{ij0}$ (for $j\geq 1$).  Thus,
$C(P^R)+C(P^S)$ equals $-\sum_{i,j\geq 0} e_{ij0}$, which
equals the negative of the number of odd integers in the matrix $X$.

Recall that $T(X) = (t_1, \ldots)$ is determined by $t_n = \sum_{i+j=n} x_{ij}$.
Since $b(X) = 1$,
the definition of $b(X)$ (see Theorem 4b of \cite{Mi}) implies
that at most one value of $x_{ij}$ is odd in each sum $\sum_{i+j=n} x_{ij}$.
Therefore, the number of odd integers in $X$ equals the number of
odd integers in $T(X)$.
In other words, 
$C(P^{T(X)}) = C(P^R) + C(P^S)$.  It follows that $u(X) = 0$.
\end{proof}

\begin{prop}
\label{prop:assoc-graded}
\mbox{}\par
\begin{enumerate}[(a)]
\item 
The tri-graded algebras $\Gr_I(A)$ and
$\Gr_{I_{\cl}}(A_{\cl})\tens_{\F_2} \F_2[\tau]$ are isomorphic.
\item The quadruply-graded rings
$\Ext_{\Gr_I(A)}(\M_2,\M_2)$ and 
$\Ext_{\Gr(A_{cl})}(\F_2,\F_2)\tens_{\F_2} \F_2[\tau]$
are isomorphic.
\end{enumerate}
\end{prop}

\begin{proof}
Part (a) follows from Lemma \ref{lem:u(X)}, as in \cite{Ma}.
It turns out that $v(R)$ equals the May filtration of $P^R$.  
The main point is that when we consider the Milnor basis relations
modulo higher May filtration, no coefficients
of $\tau$ appear.

Part (b) follows
formally from part (a), using that $\F_2 \map \F_2[\tau]$
is flat.
\end{proof}

We are interested in the quadruply-graded ring 
$\Ext_{\Gr_I(A)}(\M_2,\M_2)$.  Our convention is to grade an element
$x$ in the form $(m, s, f, w)$, where $m$ is the May filtration,
$s$ is the stem (i.e., the topological degree minus the homological degree),
$f$ is the Adams filtration (i.e., the homological degree), and $w$ is the weight.

The classical ring $\Ext_{\Gr(A_{cl})}(\F_2,\F_2)$ is studied in great detail
in \cite{Ma} and \cite{T1}, including complete information through
the 164-stem.  It can be computed as the cohomology
of the differential graded algebra
$\F_2[h_{ij}\,|\,i>0,j\geq 0]$, with 
differential given by $d(h_{ij})=\sum_{0<k<i} h_{kj}h_{i-k,k+j}$.  
The element $h_{ij}$ is dual to the Milnor basis element 
$P^R$, where $R$ consists of all zeros except for $2^j$ at the $i$th place.

By Proposition \ref{prop:assoc-graded},
the motivic ring
$\Ext_{\Gr_I(A)}(\M_2,\M_2)$ is the cohomology of the differential
graded algebra
$\F_2[\tau,h_{i,j}\,|\,i>0,j\geq 0]$ where:
\begin{enumerate}
\item
$\tau$ has degree $(0,0,0,-1)$.
\item
$h_{i0}$ has degree $(i, 2^i - 2, 1, 2^{i-1} - 1)$.
\item
$h_{ij}$ has degree $(i, 2^j(2^i-1)-1, 1, 2^{j-1}(2^i-1))$ if $j > 0$.
\end{enumerate}
As in the classical case,
these degrees follow from the description of $h_{ij}$ as the dual of a 
specific Milnor basis element.
The differential is given by the same formula as the classical one (which happens to be
homogeneous with respect to the weight), together with $d(\tau)=0$.

The  tables below list the generators and relations 
for $\Ext_{\Gr_I(A)}(\M_2,\M_2)$ in stems less than 36.
By Proposition \ref{prop:assoc-graded} 
this information can be directly lifted from the
classical computation, which is described in \cite[Theorem 1.2]{T1}.

\begin{table}[!htbp]
\caption{Generators for $\Ext_{\Gr_I(A)}(\M_2,\M_2)$}
\begin{center}
\label{table:gen}
\begin{tabular}{|l|l|l|}
\hline
generator & degree & description in terms of $h_{ij}$ \\
\hline
$\tau$  & $(0,0,0,-1)$ & \\
$h_0$ & $(1,0,1,0)$ & $h_{10}$ \\
$h_1$ & $(1,1,1,1)$ & $h_{11}$ \\
$h_2$ & $(1,3,1,2)$ & $h_{12}$ \\
$h_3$ & $(1,7,1,4)$ & $h_{13}$ \\
$h_4$ & $(1,15,1,8)$ & $h_{14}$ \\
$h_5$ & $(1,31,1,16)$ & $h_{15}$ \\
$b_{20}$ & $(4,4,2,2)$ & $h_{20}^2$ \\
$b_{21}$ & $(4,10,2,6)$ & $h_{21}^2$ \\
$b_{22}$ & $(4,22,2,12)$ & $h_{22}^2$ \\
$b_{30}$ & $(6,12,2,6)$ & $h_{30}^2$ \\
$b_{31}$ & $(6,26,2,14)$ & $h_{31}^2$ \\
$b_{40}$ & $(8,28,2,14)$ & $h_{40}^2$ \\
$h_0(1)$ & $(4,7,2,4)$ & $h_{20} h_{21} + h_{11} h_{30}$ \\
$h_1(1)$ & $(4,16,2,9)$ & $h_{21} h_{22} + h_{12} h_{31}$ \\
$h_2(1)$ & $(4,34,2,18)$ & $h_{22} h_{23} + h_{13} h_{32}$ \\
\hline
\end{tabular}
\end{center}
\end{table}

\begin{table}[!htbp]
\caption{Relations for $\Ext_{\Gr_I(A)}(\M_2,\M_2)$}
\begin{center}
\label{table:rel}
\begin{tabular}{|l|l|l|}
\hline
$h_0 h_1=0$ & 
$h_2 h_0(1) = h_0b_{21}$ &
$b_{20} b_{22} = h_0^2 b_{31} + h_3^2 b_{30}$ \\
$h_1 h_2=0$ & 
$h_3 h_0(1) = 0$ &
$b_{20} h_1(1) = h_1 h_3 b_{30}$ \\
$h_2 h_3=0$ &
$h_0 h_1(1) = 0$ & 
$b_{22} h_0(1) = h_0 h_2 b_{31}$ \\
$h_3 h_4=0$ &
$h_3 h_1(1) = h_1 b_{22}$ &
$h_0(1)^2 = b_{20} b_{21} + h_1^2 b_{30}$ \\
$h_2 b_{20} = h_0 h_0(1)$ &
$h_4 h_1(1) = 0$ &
$h_1(1)^2 = b_{21}b_{22} + h_2^2 b_{31}$ \\
$h_3 b_{21} = h_1 h_1(1)$ &
$h_1 h_2(1) = 0$ &
$h_0(1) h_1(1) = 0$ \\
\hline
\end{tabular}
\end{center}
\end{table}

If we were to draw a chart of the
motivic May $E_2$-term, analogous to our chart of the Adams $E_2$-term
in Appendix~\ref{se:motivicext}, it would be entirely in black and
look the same as the classical May $E_2$; each generator would have an
associated weight according to the formulas in Table \ref{table:gen}.
We have chosen not to give this chart
because it is the same as the classical chart and is complicated.

The $d_2$ differentials in the spectral sequence are easy to analyze.  They must
be compatible with the $d_2$ differentials in the classical May spectral sequence,
and they must preserve the weight.  The first time this is interesting
is for $b_{20}$.  Classically one has $d_2(b_{20})=h_1^3+h_0^2h_2$.
Motivically, $b_{20}$ and $h_0^2h_2$ both have weight $2$ whereas
$h_1^3$ has weight $3$.  So the motivic formula must be
$d_2(b_{20})=\tau h_1^3 + h_0^2h_2$.  Note that $h_0^ib_{20}$
kills $h_0^{2+i}h_2$, just as in the classical case, but that
$h_1^ib_{20}$ kills $\tau h_1^{3+i}$ rather than
$h_1^{3+i}$.

The following proposition lists the $d_2$ differentials on all of our generators
with stem less than 36,
modified from \cite[Theorem 2.4]{T1}.

\begin{table}[!htbp]
\label{table:d2}
\caption{$d_2$ differentials in the motivic May spectral sequence}
\begin{center}
\begin{tabular}{|l|l||l|l|}
\hline
$x$ & $d_2(x)$ & $x$ & $d_2(x)$ \\
\hline
$h_0$ & 0 &
$b_{22}$ & $h_3^3 + h_2^2 h_4$ \\
$h_1$ & 0 &
$b_{30}$ & $\tau h_1 b_{21} + h_3 b_{20}$ \\
$h_2$ & 0 &
$b_{31}$ & $h_2 b_{22} + h_4 b_{21}$ \\
$h_3$ & 0 &
$b_{40}$ & $\tau h_1 b_{31} + h_4 b_{30}$ \\
$h_4$ & 0 &
$h_0(1)$ & $h_0 h_2^2$ \\
$h_5$ & 0 &
$h_1(1)$ & $h_1 h_3^2$ \\
$b_{20}$ & $\tau h_1^3 + h_0^2 h_2$ &
$h_2(1)$ & $h_2 h_4^2$ \\
$b_{21}$ & $h_2^3 + h_1^2 h_3$ 
& & \\
\hline
\end{tabular}
\end{center}
\end{table}

Using the above table and the Leibniz rule (together with the
relations listed in Table~\ref{table:rel}), it is a routine but
tedious process to calculate the $E_3$-term.  
As in the classical case, $d_r$ is identically zero for odd $r$ for
dimension reasons.  Therefore, the $E_4$-term is equal to the $E_3$-term.
This is depicted in
Appendix~\ref{fig:mayss} below.  Note that all the differentials in the
May spectral sequence go up one box and to the left one box.

\subsection{$E_4$-term of the motivic May spectral sequence}

We now analyze the $E_4$-term of the motivic May spectral sequence.
As for the $d_2$-differentials, the $d_4$-differentials must be
compatible with the differentials in the classical May spectral sequence.
As shown in \cite{T1}, we have:
\begin{enumerate}
\item
$d_4(b_{20}^2)=h_0^4h_3$.
\item
$d_4(h_2 b_{30})=h_0^2h_3^2$.
\end{enumerate}

At this point, something new occurs.  Classically, $d_4(b_{21}^2)$ is
zero for dimension reasons: $d_2(h_1 h_4 b_{20}) = h_1^4 h_4$ and so
there is nothing in $E_4$ which $b_{21}^2$ could hit.
But motivically
we have $d_2(h_1h_4b_{20})=\tau h_1^4h_4$, and so
$h_1^4h_4$ survives to $E_4$.
Thus it is possible that
$d_4(b_{21}^2) = h_1^4 h_4$.  
Unfortunately, comparison to the classical May spectral sequence
cannot distinguish between the two possible values of $d_4(b_{21}^2)$,
since the classical comparison does not see the $\tau$-torsion.

It turns out that $d_4(b_{21}^2)$ does equal $h_1^4 h_4$.  Before
we can prove this, we must introduce a new tool.

Because $\Gr_I(A)$ is a co-commutative Hopf algebra,
$\Ext_{\Gr(A)}(\M_2,\M_2)$ has algebraic Steenrod
operations on it \cite{Ma2}.  These operations satisfy the following
properties:
\begin{enumerate}
\item
$\Sq^n(xy) = \sum_{i+j = n} \Sq^i(x) \Sq^j(y)$ (Cartan formula).
\item
$\Sq^n(x) = x^2$ if the homological degree of $x$ is $n$.
\item
$\Sq^n(x) = 0$ if the homological degree of $x$ is greater than $n$.
\end{enumerate}

In the $E_2$-term of the classical May spectral sequence, 
one has $\Sq^0(h_i)=h_{i+1}$, $\Sq^0(b_{ij})=b_{i,j+1}$,
and $\Sq^0(h_i(j))=h_{i+1}(j)$.  Recall that $\Sq^0$ preserves the
homological degree, but doubles the internal degree.

Motivically, $\Sq^0$ still preserves the homological degree but
doubles the topological degree and the weight.  By comparison to
the classical situation, we immediately obtain the following
calculations.

\begin{prop}
In $\Ext_{\Gr_I(A)}(\M_2,\M_2)$, the following formulas hold:
\begin{enumerate}[(a)]
\item
$\Sq^0(h_0)=\tau \, h_1$.
\item
$\Sq^0(h_i)=h_{i+1}$ for $i\geq 1$.
\item 
$\Sq^0(b_{i0})=\tau^2\, b_{i1}$.
\item
$\Sq^0(b_{ij})=b_{i,j+1}$ for $j\geq 1$.
\item
$\Sq^0(h_0(S))= \tau h_{1}(S)$ for all $S$.
\item
$\Sq^0(h_i(S))=h_{i+1}(S)$ for all $S$ and all $i \geq 1$.
\end{enumerate}
\end{prop}

We will use these facts to continue our analysis of differentials.

\begin{prop}
\label{pr:g}
$d_4(b_{21}^2)= h_1^4h_4$.
\end{prop}

\begin{proof}
We argue that $h_1^4h_4$ must be zero in $\Ext_{A}(\M_2,\M_2)$;
the only way it can die is for $b_{21}^2$ to hit it, and this must
happen at $d_4$ because the May filtrations of $b_{21}^2$ and
$h_1^4h_4$ are $8$ and $5$, respectively.

To see that $h_1^4h_4$ is zero in $\Ext$ over $A$, note first
that $h_2^3+h_1^2h_3$ is zero because $b_{21}$
kills it.  Now compute that $\Sq^2(h_2^3+h_1^2h_3) = h_1^4 h_4$ 
using the facts listed above,
and therefore this class must vanish. 
\end{proof}

\begin{remark}
\label{re:why-g}
We can now explain our notation $[\tau g]$ for the element in
$\Ext_A^{4,(24,11)}(\M_2,\M_2)$ (see Appendix~\ref{se:motivicext}).  The
name $g$ in some sense rightfully belongs to $b_{21}^2$.
Classically this class survives the May spectral sequence, but
motivically it does not.  Motivically, only $\tau b_{21}^2$ survives,
and this gives the generator $[\tau g]$ of $\Ext^{4,(24,11)}$.  

Even though $b_{21}^2$ does not survive the May spectral sequence,
certain multiples of it do survive.  For instance, both $h_2b_{21}^2$
and $h_3b_{21}^2$ survive.  These yield the elements of
$\Ext_A(\M_2,\M_2)$ we called $[h_2g]$ and
$[h_3g]$.
\end{remark}

Analyzing the May filtration now shows that there are no further
differentials until $E_8$, where we have $d_8(b_{20}^4)=h_0^8h_4$,
just as happens classically.  Then $E_8=E_\infty$ through the 20-stem,
and we have computed $\Ext$ over
the motivic Steenrod algebra through the 20-stem, in exact agreement
with the computer calculations.

We leave it to the interested reader to continue this analysis into
higher stems.

\begin{remark}
\label{re:h1-annihilate}
In $\Ext_A(\M_2,\M_2)$ one has $\tau h_1^4=0$; therefore
$\tau h_1^4h_n$ is zero for all $n$.  Classically,
$h_1^4h_n$ is actually {\it zero\/} for all $n$, and we have seen
that motivically this is also true when $n\leq 4$.  It is not true
that $h_1^4h_n=0$ for $n\geq 5$, however.

We do know that $h_1^{2^{n-2}}h_n=0$ for $n\geq 4$.  To
see this, take $h_1^4h_4=0$ and apply $\Sq^4$ to it to derive
$h_1^8h_5=0$; then apply $\Sq^{8}$, $\Sq^{16}$, and so on, to
inductively derive the relations for all $n$.   
In terms of our $\Ext$-chart, what this means is that for $n\geq 5$
the element $h_n$ admits
a ladder of multiplication-by-$h_1$'s which turns red at the fourth
rung ($\tau h_1^4h_n=0$), and which stops completely at the
$2^{n-2}$nd rung. It is not clear whether the ladder stops {\it
before\/} this, but we speculate that it does not.   
\end{remark}

\begin{remark}
\label{re:h1-stable}
One can consider $h_1$-stable groups associated to the motivic May 
spectral sequence.  Namely, consider the colimit of the sequence
\[
\Ext_{\Gr_I(A)}^{s,(a,b,c)} \map
\Ext_{\Gr_I(A)}^{s+1,(a+1,b+2,c+1)} \map
\Ext_{\Gr_I(A)}^{s+2,(a+2,b+4,c+2)} \map \cdots
\]
where the maps are multiplication by $h_1$.  One obtains a 
tri-graded spectral sequence which converges to information about
which elements of $\Ext_A$ support infinitely many multiplications
by $h_1$.

Calculations with the $h_1$-stable groups are much simpler than
calculations with the full motivic May spectral sequence.  For example,
one can show easily that:
\begin{enumerate}
\item
If $x$ supports infinitely many multiplications by $h_1$, then so does
$Px$.
\item
$c_0 d_0$ and $c_0 e_0$ both support infinitely many multiplications
by $h_1$.  In particular, they are non-zero.
\end{enumerate}
One consequence of the first observation is that $h_1^i P^k h_1$ is non-zero
for all $i$ and $k$.
Using that $(P^k h_1)^n = h_1^{n-1} P^{kn} h_1$,
it follows that $P^k h_1$ is not a nilpotent element in motivic stable
homotopy.
\end{remark}

\subsection{Future directions}

We close this portion of the paper with some unresolved problems.

\begin{ques}
Compute the ring $\Ext(\M_2,\M_2)[h_1^{-1}]$.
\end{ques}

Some elements, like $c_0$, survive in this localization.  Other elements,
like $h_4$, are killed.  Some elements, like $h_4 c_0$, are killed by 
different powers of $h_1$ motivically and classically.
See Remark \ref{re:h1-stable} for one possible approach to this problem.

\begin{ques}
For each $i$, determine the smallest $k$ such that $h_1^k h_i = 0$.
\end{ques}

For $h_2$, $h_3$, and $h_4$, the motivic and classical answers are the same.
Computer calculations show that $h_1^8 h_5 = 0$ but $h_1^7 h_5$ is non-zero. 
This is different than the classical situation, where $h_1^4 h_5$ is zero
but $h_1^3 h_5$ is non-zero.  For more about this, see 
Remark~\ref{re:h1-annihilate}.

\begin{ques}
Find copies of $\M_2/\tau^k$ for $k > 1$ in $\Ext(\M_2,\M_2)$.
\end{ques}

Further computer calculations show that a copy of $\M_2/\tau^2$ appears
in the 40-stem.  
Computer calculations have also spotted a copy of $\M_2/\tau^3$.


\section{Differentials, and an application to classical topology}
\label{se:diffs}

By comparison to the classical Adams spectral sequence, we immediately
obtain the differentials
\begin{align*}
d_2(h_4) & = h_0 h_3^2 \\
d_3(h_0 h_4) & = h_0 d_0 \\
d_3(h_0^2 h_4) & = h_0^2 d_0 \\
d_2(e_0) & = h_1^2 d_0 \\
d_2(f_0) & = h_0^2 e_0 \\
d_2(h_1 e_0) & = h_1^3 d_0.
\end{align*}
Through the 18-stem these are the only non-zero differentials.

The multiplicative structure implies that 
we have an infinite family of new differentials of the form
\[
d_2( h_1^k e_0 ) = h_1^{k+2} d_0.
\]

Weight considerations allow for the possibility that $d_6(c_1)$ equals
$h_1^2 Pc_0$.  The proof that this doesn't happen involves two steps.
First, $h_1^4Pc_0$ survives the spectral sequence because it cannot be
hit by anything: the class has weight $13$ and there is nothing of
weight $13$ in column $21$.  
Then $h_1^2Pc_0$ cannot be hit by $c_1$ because the latter is
annihilated by $h_1^2$ whereas the former is not.

There is also the possibility of an exotic differential
$d_5(h_2e_0)=h_1^3Pc_0$.  
Note that $h_2e_0\neq h_0[\tau g]$ (they
live in different weights) and so we may not argue that this $d_5$
vanishes using $h_0$-linearity as one would do classically.  
However, as in the previous paragraph, $h_2 e_0$ is annihilated by
$h_1$, while $h_1^3 Pc_0$ is not.

A similar analysis confirms that $[h_2g]$ does not hit $h_1^5P^2h_1$
via a $d_9$.  

The next possible $d_2$ differential is
\[
d_2( c_0 e_0 ) = h_1^2 c_0 d_0,
\]
which follows again from the multiplicative structure.

The class $h_1 h_4 c_0$ is a product of two permanent cycles $h_1 h_4$
and $c_0$, so it is a permanent cycle.  In particular, $d_3( h_1 h_4 c_0)$
equals zero, not $h_1 c_0 d_0$.  It follows that $h_4 c_0$ is also
a permanent cycle.

We now conclude that $c_0 d_0$ and $h_1 c_0 d_0$ are permanent cycles.
The class $h_1^2 h_4 c_0$ also survives to $E_\infty$, as do the
classes $h_1 [h_3 g]$ and $h_1^2 [h_3 g]$.  These are `exotic'
permanent cycles, in the sense that one does not see them in the
classical Adams spectral sequence.  In the
next sections we will remark on what this tells us about motivic stable
homotopy groups, but this will depend on convergence questions about
the spectral sequence and also on an analysis of some hidden
multiplicative $\tau$-extensions.  See Remark~\ref{re:summary}.

This essentially determines all differentials in the spectral sequence
until the $27$-stem, where one is faced with the possibility of a
$d_2$ or even a $d_7$ on the class $[h_3g]$.  The techniques we have
described so far do not let us compute this differential. 
The analysis of differentials for $t-s\leq 34$ we will be completed in
Section~\ref{se:motAN}, using information
coming from the motivic Adams-Novikov spectral sequence.

\begin{remark}
\label{re:genF}
Suppose now that the ground field $F$ is not algebraically closed.
In this case the ring $\M_2$
(the motivic cohomology of a point) is more complicated, but still
known.  The work in
\cite{V1} shows that $\M_2^{p,q}$ is nonzero only in the range
$0\leq p\leq q$, that $\M_2^{p,p}\iso K^M_p(F)/2$ (the mod $2$ Milnor
$K$-theory of $F$), that
$\M_2^{0,1}=\Z/2.\langle \tau \rangle$, and that the multiplication by
$\tau^i$ maps $\M_2^{p,p} \ra \M_2^{p,p+i}$ are isomorphisms for all
$p$ and all $i\geq 0$.  The ring $\M_2$ is generated by $\M_2^{1,1}$
and $\tau$.

It is not hard to show from this that if $F$ contains a square root of
$-1$ then the action of the Steenrod operations on $\M_2$ is trivial.  If
we let $A$ denote the motivic Steenrod algebra over $F$ and $\Abar$
the motivic Steenrod algebra over the algebraic closure, then $A\iso
\M_2\tens_{\F_2[\tau]} \Abar$.  It follows by a change-of-rings argument
that 
\[ \Ext_A(\M_2,\M_2)\iso
\Ext_{\Abar}(\F_2[\tau],\F_2[\tau])\tens_{\F_2[\tau]} \M_2.
\]  
This makes for a nice picture of the Adams $E_2$-term---it looks the
same as in Appendix~\ref{se:motivicext}, but now every copy of $\M_2$
consists of groups which not only extend down into the page (via
multiplication by elements of $\M_2^{0,1}$) but also down and to the
left (via multiplication by elements of $\M_2^{1,1}$).  This opens up
the possibility for new differentials which take a generator to some
product of a generator with an element of $\M_2^{p,q}$.  We do not
know if any of these exotic differentials actually exist.
\end{remark}

\subsection{A new tool for classical differentials}
\label{se:top1}
The comparison between the motivic Adams spectral sequence and the
classical Adams spectral sequence provides a new tool for studying
classical differentials.  The point is that the weight gives a simple
method for determining that certain classical differentials must
vanish.

We illustrate this with a basic example of $h_1 h_4$ in the
16-stem. Motivically, this element has weight 9.  Considering elements
of weight 9 in the 15-stem, we see that the only possible motivic
differential is $d_3( h_1 h_4) = h_1 d_0$.  It follows that the only
possible classical differential on $h_1 h_4$ is $d_3(h_1 h_4) = h_1
d_0$.  

The same argument yields the following, as well as many other results
similar to it:

\begin{prop}
\label{pr:h1hj}
In the classical Adams spectral sequence there are no differentials
of the form $d_r(h_1h_j)=h_0^{n}h_j$, for any $j$ and $r$.  
\end{prop}

The above result is already known, of course (in fact it is known that
$h_1h_j$ is a permanent cycle, which is stronger).  Still, the
simplicity of our motivic proof is appealing.

The most famous question about the vanishing of differentials in the
classical Adams spectral sequence is the Kervaire problem, which asks
whether all differentials vanish on the classes $h_i^2$.  One might
hope that our weight arguments say something nontrivial about this
problem, but unfortunately this doesn't seem to be the case.  
For example, consider the element $h_4^2$ in the 30-stem.  This element 
has weight 16.  But the elements $k$, $h_0 k$, and $h_0^2 k$ in the
29-stem all have weight 16, so we learn nothing by considering the weights.
Although it is not shown on our chart, it turns
out that something similar happens with $h_5^2$ in the 62-stem,
where $h_5^2$ has weight 32 and so do almost all of the generators in the
61-stem.


\section{Convergence issues}
\label{se:conv}
The evident vanishing line in the motivic Adams spectral sequence
shows that it is certainly converging to {\it something\/}; but we
would like to give this `something' an appropriate homological name.
In this section we will prove that the spectral sequence converges to
the homotopy groups of the $H$-nilpotent completion of the motivic
sphere spectrum $S$.  This is
almost by definition, but not quite---one must be careful of details.

\medskip

\subsection{Background}
Below we will set up the homological Adams spectral sequence in the
usual way, using the geometric cobar resolution.  But in order for
this to work we need to know $H_{*,*}(H)$ and, more generally,
$H_{*,*}(H\Smash H\Smash \cdots \Smash H)$.  The former is essentially
calculated by Voevodsky in \cite{V2,V3}, but does not appear
explicitly.  For the latter, there is no general K\"unneth theorem for
motivic homology and cohomology.  So in both cases we must work a
little to derive the necessary results.

First recall that since $H$ is a ring spectrum, there is a canonical
map $\theta_X\colon H_{**}(X) \ra \Hom_{\M_2}(H^{**}(X),\M_2)$ for any motivic
spectrum $X$.  Given a homology element $S^{p,q}\ra H\Smash X$ and a
cohomology element $X\ra S^{s,t}\Smash H$, one forms the composite
\[ S^{p,q}\ra H\Smash X \ra H \Smash S^{s,t}\Smash H \ra S^{s,t}\Smash
H\Smash H \ra S^{s,t} \Smash H
\]
to get an element of $\pi_{p-s,q-t}(H)$.  This pairing is
$\M_2$-linear in the cohomology class, and so induces the desired map.

\begin{prop}
\label{pr:HH}
The map $\theta_H\colon H_{**}(H) \ra \Hom_{\M_2}(H^{**}(H),\M_2)$ is
an isomorphism.
\end{prop}

Since $A=H^{**}(H)$ this
identifies $H_{**}(H)$ with the dual steenrod algebra $A_*$ as
computed by Voevodsky in \cite{V2}.  

To prove the above proposition we work in the category of
$H$-modules.  
By \cite[Thm. 1]{RO}, the homotopy category of $H\Z$-modules is
equivalent to Voevodsky's triangulated category of motives.  A similar
result holds for $H$-modules, where one works with motives having
coefficients in $\F_2$.  A \dfn{Tate motive}, under this equivalence,
is an $H$-module equivalent to a wedge of modules of the form
$\Sigma^{p,q}H$.  A \mdfn{proper Tate motive of weight $\geq n$}, 
in the sense of \cite[Def. 3.49]{V3}, is a Tate motive equivalent to a wedge of
modules of the form $\Sigma^{p,q}H$ where $p\geq 2q$ and $q\geq n$.  The result
\cite[Thm. 4.20(2)]{V3} states that the spectra $H\Smash \Sigma^\infty
K(\Z/2(n),2n)$ are proper Tate motives of weight $\geq n$.  From this
we deduce the following:

\begin{lemma}
\label{le:HH=Tate}
As a left $H$-module, $H\Smash H$ is equivalent to a motivically finite type wedge of
$H$ (with one summand for every admissible monomial
in the motivic Steenrod algebra).  In particular, $\pi_{**}(H\Smash H)$ is free
as an $\M_2$-module.  
\end{lemma}

\begin{proof}
As a motivic spectrum, $H$ is  a directed homotopy colimit of desuspensions
of the spectra $\Sigma^\infty K(\Z/2(n),2n)$. Precisely,
\[ H\he \hocolim_n \Sigma^{-2n,-n} \Sigma^\infty K(\Z/2(n),2n),
\]
where the maps in this hocolim come from the structure maps in the
spectrum $H$.  It follows that $H\Smash H$ is equivalent to
\begin{myequation}
\label{eq:hoco}
 \hocolim_n \bigl [\Sigma^{-2n,-n} H\Smash \Sigma^\infty K(\Z/2(n),2n)
\bigr ],
\end{myequation}
where the homotopy colimit is taken in the category of left
$H$-modules.  Since each $H\Smash \Sigma^\infty K(\Z/2(n),2n)$ is a
proper Tate motive of weight $\geq n$, its $\Sigma^{-2n,-n}$
desuspension is a proper Tate motive of weight $\geq 0$.  By
\cite[Prop. 3.65]{V3}, the homotopy colimit of (\ref{eq:hoco}) is therefore also a
proper Tate motive of weight $\geq 0$.  

To complete the proof we consider the chain of isomorphisms
\[ A^{**}=[H,\Sigma^{**}H]\iso [H\Smash
H,\Sigma^{**}H]_H=[\Wedge_\alpha
\Sigma^{n_\alpha}H,\Sigma^{**}H]_H=\prod_\alpha
\Sigma^{n_\alpha}H^{**}.
\]
Here $[\blank,\blank]_H$ denotes maps in the homotopy category of left
$H$-modules.  The fact that $A$ is free on the admissible monomials
tells us what the $n_\alpha$'s must be (and that the above product is also a
sum).
\end{proof}

\begin{proof}[Proof of Proposition~\ref{pr:HH}]
If $M$ is an $H$-module, there is a canonical map
\[ \theta_{H,M}\colon \pi_{*,*}(M) \llra{\iso} [\Sigma^{**}H,M]_H \ra \Hom_{\M_2}(
[\Sigma^{**}M,H]_H, \M_2).
\]
The definition of the second map in the composite is completely analagous to the
definition of $\theta_X$ above.  Moreover,
when $M$ has the form $H\Smash X$, the map $\theta_{H,M}$ is naturally
isomorphic to $\theta_X$.  So our goal is to show that
$\theta_{H,H\Smash H}$ is an isomorphism.

It is clear that $\theta_{H,M}$ is an isomorphism when $M=H$, or more
generally when $M$ is a motivically finite type wedge of suspensions of $H$.  
So by Lemma~\ref{le:HH=Tate} this applies for $M=H\Smash H$.
\end{proof}

Our next goal is the following K\"unneth isomorphism:

\begin{prop}
\label{pr:kunneth}
For any motivic spectrum $A$ which admits a right $H$-module structure, the natural
K\"unneth map $H_{**}(A)\tens_{\M_2} H_{**}(H)\ra H_{**}(A\Smash H)$
is an isomorphism.  In particular, the maps $
H_{**}(H)^{\tens(s)} \ra H_{**}(H^{\wedge(s)})$ are
isomorphisms for every $s$ (where the tensors are over $\M_2$).  
\end{prop}

This result is automatic if one believes that $H$ is cellular in the
sense of \cite{DI2}, using the standard K\"unneth isomorphism for
cellular objects [loc. cit].  The cellularity of $H$ has been claimed
by Hopkins and Morel, but their proof has never appeared.  We will
avoid this issue by instead proving the proposition in a way
that only uses the
cellularity of $H\Smash H$ as an $H$-module.

Consider the homotopy category of $H$-modules.  Define the
\dfn{$H$-cellular spectra} to be the smallest full subcategory which
contains $H$, is closed under arbitrary wedges, and has the property
that when  $A\ra B \ra C$ is a cofiber sequence and two of the three
terms are in the subcategory, so is the third.  By Lemma~\ref{le:HH=Tate},
$H\Smash H$ is $H$-cellular as a left $H$-module.

\begin{proof}[Proof of Proposition~\ref{pr:kunneth}]
If $A$ is any right $H$-module, then $A\Smash H\he A\Smash_H (H\Smash
H)$.  Since $H\Smash H$ is a wedge of suspension of $H$'s, $A\Smash H$
is therefore a wedge of suspensions of $A$'s.  It's easy to now show
that $H_{**}(A\Smash H)\iso H_{**}(A)\tens_{\M_2} H_{**}(H)$.  
\end{proof}

Let $\SSH$ be the smallest full subcategory of the motivic stable homotopy
category that satisfies
the following properties:
\begin{enumerate}[(i)]
\item $\SSH$ contains $S$;
\item If $A\ra B\ra C$ is a cofiber sequence and two of the terms
belong to $\SSH$, then so does the third;
\item $\SSH$ is closed under arbitrary wedges;
\item If $A\in \SSH$ then $A\Smash H$ is also in $\SSH$.  
\end{enumerate}
The smallest full subcategory of the motivic stable homotopy category
that satisfies (i)--(iii)
is denoted $\SS$, and equals the category of cellular spectra of \cite{DI2}.

\begin{lemma}
\label{le:kunneth}
For every $A\in \SSH$, the canonical maps 
$H_{**}(A)\tens_{\M_2}
H_{**}(H^{\wedge(s)}) \ra H_{**}(A\Smash H^{\wedge(s)})$ are
isomorphisms for every $s$.
\end{lemma}

\begin{proof}
Let $\cD$ be the full subcategory of the motivic stable homotopy category
consisting of all
objects $A$ for which the maps in the statement of the lemma are
isomorphisms, for all $s$.  It is evident that $\cD$ satisfies
properties (i)--(iii).  The second statement of
Proposition~\ref{pr:kunneth} implies that $\cD$ satisfies (iv)
as well, so $\cD$ contains $\SSH$.
\end{proof}

\subsection{The homological Adams spectral sequence}
In this section we set up homological Adams spectral sequences in the
motivic world, and discuss their convergence properties.  Let $E$ be a
motivic homotopy ring spectrum (a ring object in the homotopy category
of motivic spectra).  Let $\bE$ be the homotopy fiber of $S\ra E$.  As
in
\cite{Bo}, for any spectrum $X$ there is a standard tower of 
homotopy fiber sequences having
\[ X_s=\bE^{\Smash(s)}\Smash X, \qquad W_s=E\Smash X_s=E\Smash
\bE^{\Smash(s)}\Smash X.\]
The fiber sequence $X_{s+1}\ra X_s\ra W_s$ is induced by smashing 
$\bar{E}\ra S\ra E$ with $X_s$.  This tower satisfies conditions (a-c) in 
\cite[Def. 2.2.1]{R} for being an $E_*$-Adams resolution.  Below we
will show that it also satisfies condition (d).

Let $C_{s-1}$ be the homotopy cofiber of $X_s \ra X_0$.  Then there
are induced maps $C_s\ra C_{s-1}$, and the cofiber of this map is
$\Sigma^{1,0} W_s$.  One gets a tower under $X$ of the form
\[ \xymatrix{
&& \Sigma W_3 & \Sigma W_2 & \Sigma W_1 \\
X \ar[r] & \cdots \ar[r] & C_2\ar[r]\ar[u] & C_1 \ar[u]\ar[r] & C_0.\ar[u]
}
\]
The homotopy limit of the $C_i$'s is called the {\it $E$-nilpotent
completion\/} of $X$, and denoted $X^{\wedge}_E$ (see \cite[Section 5]{Bo}).
Note that for formal reasons the homotopy spectral sequences of the
$\{C_s\}$ tower and the $\{X_s\}$ tower may be identified.

Now suppose that $E$ is associative and unital on-the-nose---that is,
$E$ is a motivic symmetric ring spectrum.  Then for any spectrum $X$ one
may consider the cosimplicial spectrum
\begin{myequation}
\label{eq:Ecompletion}
\xymatrix{
E\Smash X \ar@<0.5ex>[r]\ar@<-0.5ex>[r] & E\Smash E\Smash X
 \ar@<0.6ex>[r]\ar[r]\ar@<-0.6ex>[r]
 & E\Smash E\Smash E\Smash X \ar@<0.8ex>[r]\ar@<0.3ex>[r]\ar@<-0.2ex>[r]\ar@<-0.7ex>[r]
 & \cdots
}
\end{myequation}
Here the coface maps are all induced by the unit $S\ra E$, and the
codegeneracies (not drawn) all come from multiplication $E\Smash E\ra E$.  The
homotopy limit of this cosimplicial spectrum is another model for
$X^{\wedge}_E$.  To see this, note that the usual $\Tot$ tower is an
$E$-nilpotent resolution of $X$ in the sense of Bousfield
\cite[Def. 5.6]{Bo}, and so the holim of this $\Tot$ tower is homotopy
equivalent to $X^{\wedge}_E$ by \cite[Prop. 5.8]{Bo}.  [Note: The
proof of \cite[Prop. 5.8]{Bo} goes through almost verbatim in the
motivic category except for one step, in \cite[5.11]{Bo}, where a map
is proven to be a weak equivalence by showing that it induces
isomorphisms on stable homotopy groups.  Motivically one must look at
homotopy classes of maps from all smooth schemes, not just spheres;
but otherwise the argument is the same.] 

Now, it is a general fact in any model category that a homotopy limit
of a cosimplicial object is weakly equivalent to the homotopy limit of
the corresponding diagram in which one forgets the codegeneracies.
This is because the subcategory of $\Delta$ consisting of the
monomorphisms is homotopy initial inside of $\Delta$.  But if $X$ is itself
a ring spectrum, then the cosimplicial object $E^{\bullet}\Smash X$
(regarded without codegeneracies) is a diagram of ring spectra.  We
may then take its homotopy limit {\it in the model category of motivic
ring spectra\/}, and we get something whose underlying spectrum is
equivalent to what we got by just taking the holim in motivic spectra.
This is a long way of saying that if $X$ and $E$ are both ring spectra
then via this holim we may build a
model for $X^{\wedge}_E$ which is itself a ring spectrum.

\begin{remark}
\label{re:module_completion}
If $E$ is a ring spectrum and $X$ is a module over $E$, then the
cosimplicial object of (\ref{eq:Ecompletion}) admits a contracting
homotopy.  This shows that the canonical map $X\ra X^{\wedge}_E$ is an
equivalence in this case.  We will make use of this fact in
Section~\ref{se:motAN} below.
\end{remark}

\medskip

Now we will specialize and take $E=H$.  Let $X$ be any motivic
spectrum, and write $\{X_s,W_s\}$ for the canonical Adams tower just
constructed.  Note that if $X\in \SSH$ then by a simple induction we
have that all the objects $X_s$ and $W_s$ belong to
$\SSH$ as well.

\begin{prop}
\label{pr:homconv}
For any motivic spectrum $X\in \SSH$, the $E_2$-term of the homotopy
spectral sequence for the tower $\{X_s,W_s\}$ is naturally isomorphic
to $\Ext_{H_{**}H}(\M_2,H_{**}(X))$ (where this means $\Ext$ in the category of
$H_{**}H$-comodules).
\end{prop}

\begin{proof}
The object $W_s$ is equal to $H\Smash \bH^{\wedge(s)}\Smash X$. 
It follows from Proposition~\ref{pr:kunneth} and 
an easy induction on $s$ that
\[ \pi_*(W_s)\iso [H_{**}(\bH)]^{\tens(s)} \tens_{\M_2}H_{**}(X) \]
and
\[ H_{**}(W_s)\iso H_{**}{H}\tens_{\M_2} [H_{**}(\bH)]^{\tens(s)}
\tens_{\M_2} H_{**}(X).
\]
The usual arguments show that
\[ 0 \ra H_{**}(X) \ra H_{**}(W_0) \ra H_{**}(W_1) \ra H_{**}(W_2) \ra
\cdots \]
may be identified with the cobar resolution for $H_{**}(X)$, and that
the complex
\[ 0 \ra \pi_{**}(W_0) \ra \pi_{**}(W_1) \ra \cdots \]
may be identified with the complex obtained by applying
$\Hom_{H_{**}H}(H_{**},\blank)$ to the cobar resolution.
\end{proof}

\begin{remark}
\label{re:converge}
The convergence of the homological Adams spectral sequence works in
the standard way, as described in \cite[Section 6]{Bo}.  That is, if
$\lim^{1}_r E_r^{s,t,u}(X)=0$ for each $s,t,u$ then the two natural
maps
\[ \pi_{*,*}(X^{\wedge}_H) \ra \lim_s \pi_{*,*}(C_s),
\qquad\text{and} \]
\[ F_s\pi_{*,*}(X^{\wedge}_H)/F_{s+1}\pi_{*,*}(X^{\wedge}_H)  \ra E_{\infty}^{s,*,*}(X)
\]
are isomorphisms.  Here $F_s$ denotes the filtration on
$\pi_{*,*}(X^{\wedge}_H)$ coming from the tower $\{C_t\}$.  
\end{remark}

\subsection{Comparison of towers}
\label{se:compare}
To complete our discussion we will verify that the cohomological Adams
spectral sequence for the sphere spectrum, as constructed in
Section~\ref{se:motadams}, also converges to the homotopy groups of
$S^{\wedge}_H$.  

Write $\{X'_s,W'_s\}$ for the `naive' Adams tower constructed as in
Section~\ref{se:motadams}.  That is, $X_0'=X$, each $W_s'$ is a wedge
of Eilenberg-MacLane spectra, and $X_s'\ra W_s'$ is surjective on
$H$-cohomology.  Our goal is to identify the homotopy spectral
sequences of $\{X'_s\}$ and $\{X_s\}$, under suitable assumptions on
$X$.  The latter is the same as for $\{C_s\}$, and hence
(conditionally) converges to the homotopy groups of $X^{\wedge}_H$.

\begin{lemma}
Suppose that in the tower $\{X_s',W_s'\}$, all the spectra $W_s'$ are
motivically finite type wedges of $H$.  
Suppose also that $H^{**}(X)$ is
free over $\M_2$ and the natural map $H_{**}(X)\ra
\Hom_{\M_2}(H^{**}(X),\M_2)$ is an isomorphism.  Then the complex
\[ 0 \ra H_{**}(X) \ra H_{**}(W_0') \ra H_{**}(W_1') \ra H_{**}(W_2')
\ra \cdots 
\]
is a resolution of $H_{**}(X)$ by relative injective comodules (see
\cite[A1.2.10]{R} for terminology).  Moreover, the $E_2$-term of the
homotopy spectral sequence for $\{X_s',W_s'\}$ is naturally isomorphic
to $\Ext_{H_{**}H}(\M_2,\M_2)$.  
\end{lemma}

\begin{proof}
Start with the complex
\[ 0 \la H^{**}(X) \la H^{**}(W_0') \la H^{**}(W_1') \la \cdots \]
Because of the way the tower $\{X_s',W_s'\}$ was constructed, this is
a resolution.  The assumption that each $W_s'$ is a motivically finite type
wedge of
Eilenberg-MacLane spectra shows that it is a resolution by free
$H^{**}H$ modules.  In particular, it is a resolution by free
$\M_2$-modules.  Since $H^{**}(X)$ is itself free as an $\M_2$-module,
the resolution is split (as a complex of $\M_2$-modules).  

Now apply the functor $\Hom_{\M_2}(\blank,\M_2)$.  The resulting
complex is still exact, and there is a map of complexes
\[ \xymatrixcolsep{1pc}\xymatrix{
0 \ar[r] & H_{**}(X) \ar[r]\ar[d] & H_{**}(W_0') \ar[r]\ar[d] & H_{**}(W_1')
\ar[r]\ar[d] & \cdots \\
0 \ar[r] & \Hom(H^{**}(X),\M_2) \ar[r] & \Hom(H^{**}(W_0'),\M_2)
\ar[r] 
& \Hom(H^{**}(W_1'),\M_2)
\ar[r] & \cdots 
}
\]
Each vertical map is an isomorphism, by our assumptions on $X$ and the
$W_i'$'s.  Hence the top complex is a resolution which is split over
$\M_2$.  Each $H_{**}(W_i')$ is a direct product of copies of
$H^{**}H$, and so is certainly a relative injective.

The final claim of the lemma results from the natural maps
\[ \pi_{**}(W_s') \ra \Hom_{H_{**}H}(H_{**},H_{**}(W_s')), \]
which are isomorphisms given our assumptions about the spectra $W_s'$.
\end{proof}

Our final task is to compare the $\{X_s,W_s\}$ tower to the
$\{X_s',W_s'\}$ tower.
Note that each $W_s'$ is naturally an $H$-module (being a wedge of
suspensions of $H$'s).  Using this, one can inductively
construct a map of towers $\{X_s,W_s\}\ra \{X'_s,W_s'\}$ by the
standard method.  Start with the identity map  $X_0\ra X_0'$, and
consider the diagram
\[ \xymatrix{ W_0\ar@{=}[r] & X_0\Smash H \ar[r] & X_0'\Smash H \ar[r]
& W_0'\Smash H \ar@/^2ex/[d] \\
& X_0 \ar[u]\ar[r] & X_0'\ar[u] \ar[r] & W_0'.\ar[u] \\
}
\]
One obtains a map $W_0\ra W_0'$ by following around the diagram.
There is then an induced map $X_1\ra X_1'$, and one continues
inductively.

\begin{prop}
Assume that $X$ is in $\SSH$, that $H^{**}(X)$ is
free over $\M_2$, that all the spectra $W_s'$ are motivically finite
type wedges of $H$, and that the natural map $H_{**}(X)\ra
\Hom_{\M_2}(H^{**}(X),\M_2)$ is an isomorphism.  Then the map of
towers
$\{X_s,W_s\} \ra \{X'_s,W_s'\}$ induces a map of spectral sequences
which is an isomorphism from the $E_2$-terms onward.
\end{prop}

\begin{proof}
This is almost immediate from the previous results.  The map of towers
induces a map of complexes 
\[ \xymatrix{
0\ar[r] & H_{**}(X) \ar[r]\ar[d] & H_{**}(W_0) \ar[r]\ar[d] & H_{**}(W_1)
\ar[r]\ar[d] & \cdots \\
0\ar[r] & H_{**}(X) \ar[r] & H_{**}(W_0') \ar[r] & H_{**}(W_1')
\ar[r] & \cdots 
}
\]
which is the identity on $H_{**}(X)$.  By our previous results, these
complexes are resolutions of $H_{**}(X)$ by relative injectives over
$H_{**}H$.  Hence upon applying the functor
$\Hom_{H_{**}H}(H_{**},\blank)$, the map of resolutions induces an
isomorphism on cohomology groups.  These cohomology groups are
naturally isomorphic to the $E_2$-terms of the spectral sequences for the
two respective towers, so we are done.
\end{proof}

\begin{cor}
\label{co:conv}
The cohomological motivic Adams spectral sequence for the sphere spectrum $S$
strongly converges to the bigraded homotopy groups of $S^{\wedge}_H$.  
\end{cor}

\begin{proof}
Let $X=S$.  By the above results, our naively-constructed Adams
spectral sequence for $S$ is the same (from $E_2$ on) as the homotopy
spectral sequence for the tower $\{C_s\}$.  The homotopy limit of this
tower is $S^{\wedge}_H$, and the convergence properties are as
described in Remark~\ref{re:converge}.  The evident vanishing line of
the motivic Adams spectral sequence---which can be proven in exactly
the same way as for the classical case \cite{A}---guarantees strong
convergence.
\end{proof}

Applying topological realization to the tower $\{C_s\}$ for the
motivic sphere
spectrum yields the corresponding tower for the classical sphere
spectrum.  It follows that there is a map from the topological
realization of $S^{\wedge}_{H}$ to the spectrum
$(S_{\cl})^{\wedge}_{H_{\cl}}$.  This can even be constructed as a map of
ring spectra.  We get an induced ring map 
\[ \psi\colon \pi_{*,*}(S^{\wedge}_H) \ra
\pi_*(S^{\wedge}_{H_{\cl}})=\pi_*(S)^{\wedge}_2
\]
from $\pi_{*,*}(S^{\wedge}_H)$ to
the classical $2$-completed stable homotopy groups of spheres.

\begin{remark}
\label{re:summary}
We can now summarize some of our conclusions from
Section~\ref{se:diffs} as follows.  Through the 34-stem, the elements $h_1^n$,
$h_1^nP^k{h_1}$ ($n\geq 4$), $h_1^nP^kc_0$ ($n\geq 2$), and
$h_1^2h_4c_0$ all survive the spectral sequence to determine `exotic'
motivic homotopy classes in $\pi_{**}(S^{\wedge}_H)$, in the sense
that they vanish under the map $\psi$.
These claims exactly match what was discovered in \cite{HKO} using the
Adams-Novikov spectral sequence.
The reader may wonder why we are not making the same assertion about
the classes $c_0d_0$ and $[h_3g]$.  It turns out that the latter does
not survive the spectral sequence, and for $c_0d_0$ there is a hidden
$\tau$-extension which shows that this class is not truly exotic.
 This will all be explained
at the end of the next section.  
\end{remark}

\section{The motivic Adams-Novikov spectral sequence}
\label{se:motAN}

The motivic version of the Adams-Novikov spectral sequence was
considered in \cite{HKO}.   We wish to use it here to deduce some
differentials in the motivic Adams spectral sequence, which will in turn lead
to an application to classical algebraic topology.  Our approach for
setting up the Adams-Novikov spectral sequence is somewhat different
from that of \cite{HKO}, and does not depend on any of the results from
that paper.

\medskip

\subsection{Backround}
Recall that $\MGL$ denotes the motivic cobordism spectrum.  This is
constructed as a motivic symmetric ring spectrum in \cite[Section 2.1]{PPR}.

By the same methods as in classical topology one can write down the
Quillen idempotent and use it to split off $\BPL$ from $\MGL$.  This
is remarked briefly in \cite[Section 2]{HK} and carried out in detail
in \cite{Ve}.  One constructs the Quillen idempotent $e\colon
\MGL_{(2)}\ra \MGL_{(2)}$ and then 
the motivic spectrum $\BPL$ is defined to be the homotopy colimit of the sequence
\[ \MGL_{(2)}\llra{e} \MGL_{(2)} \llra{e} \cdots 
\]

One knows from the motivic Thom isomorphism that $H^{**}(\MGL)\iso
\M_2[b_1,b_2,\ldots]$ where $b_i$ has bidegree $(2i,i)$, exactly as in
classical topology.  It follows readily that $H^{**}(\BPL)\iso
\M_2[v_1,v_2,\ldots]$ where $v_i$ has bidegree $(2(2^i-1),2^i-1)$.  By
the same arguments as in classical topology (or by using topological
realization), it is easily shown that the map $H_{**}(\BPL)\ra
H_{**}(H)$ is an injection whose image is the sub-coalgebra
$P_*=\M_2[\tau^{-1} \zeta_1^2, \tau^{-1} \zeta_2^2,\ldots]$ of $A_{**}$ (see
Section~\ref{se:milnor} for notation).

In some sense our discussion below would be most natural if there
exists a model for $\BPL$ which is a symmetric ring spectrum, and
where there is a map of symmetric ring spectra $\BPL\ra H$.  This is
not known, however, and it is stronger than what we actually need.  It
turns out that all we really need are pairings $\BPL\Smash \BPL \ra
\BPL$ and $\BPL\Smash H \ra H$ which are unital on the
nose, and that is easily arranged for reasons we now explain.

First, one observes that the Quillen idempotent $e\colon \MGL_{(2)}\ra
\MGL_{(2)}$ is a map of homotopy ring spectra (not a map of 
symmetric ring spectra).  It is then easy to see that there is a map
$\mu\colon \BPL\Smash \BPL\ra \BPL$ making $\BPL$ a homotopy ring spectrum (not
a symmetric ring spectrum) and such that $\MGL_{(2)}\ra \BPL$ is a map
of homotopy ring spectra.  Choose a model for $\BPL$ which is fibrant
as a symmetric spectrum, and for which the
unit (up to homotopy) of $\mu$ is a cofibration $S\ra \BPL$.  The map
$(S\Smash \BPL)\Wedge(\BPL\Smash S)\ra \BPL\Smash \BPL$ is then a
cofibration, and so has the homotopy lifting property.  The triangle
\[ \xymatrix{ (S\Smash \BPL)\Wedge (\BPL\Smash S) \ar[r]\ar[d] &
\BPL \\
\BPL\Smash \BPL \ar[ur]_\mu
}
\]
commutes up to homotopy, and therefore we may choose another
multiplication $\mu'\colon \BPL\Smash \BPL\ra \BPL$, homotopic to
$\mu$, so that the triangle commutes on the nose.

Analogously, start with the Thom class $\BPL\ra H$ (induced from the
Thom class on $\MGL$) and consider the composite map $\BPL\Smash H \ra
H\Smash H \ra H$.  This pairing is unital up to homotopy, and
therefore by the same arguments as above we can choose a homotopic
pairing which is unital on the nose.  

As a final prelude for the work in the next section, suppose $E$ is a
motivic spectrum with a pairing $E\Smash E\ra E$ that is unital on the
nose (but where no associativity assumptions are made).  Then for any
motivic spectrum $X$ one may form the cosimplicial object without
codegeneracies
\begin{myequation}
\label{eq:codeg}
\xymatrix{
 E\Smash X \ar@<0.5ex>[r]\ar@<-0.5ex>[r] & E\Smash E\Smash X
 \ar@<0.6ex>[r]\ar[r]\ar@<-0.6ex>[r]
 & E\Smash E\Smash E\Smash X 
\ar@<0.8ex>[r]\ar@<0.3ex>[r]\ar@<-0.2ex>[r]\ar@<-0.7ex>[r]
 & \cdots
}
\end{myequation}
The homotopy limit of this diagram may be modelled by a $\Tot$-like
object, defined to be the mapping space from the standard cosimplicial
simplex (without codegeneracies) $\Delta^\bullet$.  By looking at the
usual skeletal filtration of $\Delta^\bullet$ one defines $\Tot_i$
objects as usual, and obtains a tower of fibrations whose limit is
$\Tot$.  The homotopy fiber of $\Tot_i \ra \Tot_{i-1}$ is
$\Omega^i(E^{\wedge(i)}\Smash X)$.  

The same arguments as usual (meaning in the true cosimplicial case,
where there {\it are\/} codegeneracies) show that this $\Tot$-tower is
an $E$-nilpotent resolution of $X$ in the sense of Bousfield
\cite[Def. 5.6]{Bo}, and therefore our $\Tot$ object is a model for
the $E$-nilpotent completion of $X$.  

Finally, assume that there is a pairing $E\Smash X\ra X$ which is
unital on the nose (but where again there is no associativity
condition).  This map then gives a ``contracting homotopy'' in the
cosimplicial object without codegeneracies (\ref{eq:codeg}).  The
usual arguments show that $X$ splits off every piece of the
$\Tot$-tower, and its complementary summand in each $\Tot_i$ becomes
null homotopic in $\Tot_{i-1}$.  One shows that consequently the map
$X\ra \Tot(E^\bullet\Smash X)$ is a weak equivalence.  In other words,
the same thing that works for cosimplicial objects in the case of
pairings which are associative on the nose (see
Remark~\ref{re:module_completion}) also works for cosimplicial objects
without codegeneracies.

\subsection{Construction of the spectral sequence}
For any spectrum $X$, construct a bi-cosimplicial spectrum---without
codegeneracies---of the form
\begin{myequation}
\label{eq:bicosimp}
[n],[k]\mapsto \BPL^{\Smash(n+1)}\Smash X \Smash H^{\Smash(k+1)}.
\end{myequation}
We will continue to say ``bi-cosimplicial object'' below, even though
this is an abuse of terminology since there are no codegeneracies.
The homotopy limit of this bi-cosimplicial spectrum is a kind of
``bi-completion'', which we will denote by $X^{\wedge}_{\{\BPL,H\}}$.
We will only concern ourselves with $X=S$.

We can compute the homotopy limit of our bi-cosimplicial object by
first taking homotopy limits in one direction---with respect to $n$,
say---and then taking the homotopy limit of the resulting cosimplicial
object in the other direction ($k$ in this case).  Taking
homotopy limits with respect to $n$, one obtains the following
cosimplicial spectrum:
\begin{myequation}
\label{eq:HBP} 
\xymatrix{
H^{\wedge}_{\BPL} \ar@<0.5ex>[r]\ar@<-0.5ex>[r] & (H\Smash H)^{\wedge}_{\BPL}
 \ar@<0.6ex>[r]\ar[r]\ar@<-0.6ex>[r]
 & (H^{\Smash(3)})^{\wedge}_{\BPL} 
\ar@<0.8ex>[r]\ar@<0.3ex>[r]\ar@<-0.2ex>[r]\ar@<-0.7ex>[r]
 & \cdots
}
\end{myequation}
But $H$ is a $\BPL$-module up to homotopy, and therefore so is $H^{\Smash(k)}$ for
each $k$.  It follows from the
observations in the previous section that
the maps $H^{\Smash(k)} \ra (H^{\Smash(k)})^{\wedge}_{\BPL}$ are all equivalences.
 Therefore the homotopy limit of (\ref{eq:HBP}) is
equivalent to $S^{\wedge}_H$.  So we have shown that
$S^{\wedge}_{\{\BPL,H\}}\he S^{\wedge}_H$.

Now, we can also compute the homotopy limit of our bi-cosimplicial object by
first taking homotopy limits in the $k$ direction, and then in the $n$
direction.
Taking
homotopy limits with respect to $k$, one obtains the following
cosimplicial spectrum:
\begin{myequation}
\label{eq:BPH} 
\xymatrix{
\BPL^{\wedge}_H \ar@<0.5ex>[r]\ar@<-0.5ex>[r] & (\BPL\Smash \BPL)^{\wedge}_H
 \ar@<0.6ex>[r]\ar[r]\ar@<-0.6ex>[r]
 & (\BPL^{\Smash(3)})^{\wedge}_H 
\ar@<0.8ex>[r]\ar@<0.3ex>[r]\ar@<-0.2ex>[r]\ar@<-0.7ex>[r]
 & \cdots
}
\end{myequation}

For each $k$, we may run the homological Adams spectral sequence for
computing the homotopy groups of $(\BPL^{\Smash(k)})^{\wedge}_H$.
Since $\BPL$ is cellular, Proposition~\ref{pr:homconv} says that the
$E_2$-term of this spectral sequence has the expected description in
terms of $\Ext$.
Then by the same arguments as in classical topology, the 
spectral sequence for computing
$\pi_{**}((\BPL^{\Smash(k)})^{\wedge}_H)$ is concentrated in even
degrees, and therefore collapses.  For $k=1$ this allows us to
conclude that
\[ \pi_{**}(\BPL^{\wedge}_H) \iso \Z_{(2)}[\tau,v_1,v_2,\ldots], 
\]
and if we call the above ring $V$ then for all $k\geq 1$ we get
\[ \pi_{**}((\BPL^{\Smash(k)})^{\wedge}_H) \iso
V[t_1,t_2,\ldots]\tens_V \cdots \tens_V V[t_1,t_2,\ldots]
\quad\text{($k-1$ factors).}
\]
where each $t_i$ has bidegree $(2(2^i-1),2^i-1)$.  
In other words, everything is exactly as in classical topology except
for the extra generator $\tau$ and the new grading by weight.  

Now we run the homotopy spectral sequence for the homotopy limit of
(\ref{eq:BPH}), which from now on we will refer to as {\it the motivic
Adams-Novikov spectral sequence\/}.  We find that the $E_2$-term is
precisely $\Ext_{BP_*BP}(BP_*,BP_*)$ tensored over $\Z_{(2)}$ with
$\Z_{(2)}[\tau]$, with the generators assigned appropriate weights.  
Unlike the motivic Adams $E_2$-term,
the assignment of the weights in this case follows a very simple pattern.
Namely, generators in classical $\Ext^{s,t}$ yield motivic generators
in $\Ext^{s,(t,t/2)}$.  

A chart showing the Adams-Novikov $E_2$-term is given in
Appendix~\ref{se:AN}.  
Recall that this is really a two-dimensional
picture of a three-dimensional spectral sequence.  
We have not written any of the weights in the chart, because these are
all deduced by the simple formula given in the last paragraph.
In terms of the usual naming conventions,
$\alpha_i$ has weight $i$, $\beta_2$ has weight $5$, $\eta_{3/2}$ has
weight $13$, etc.

\subsection{Comparison of the Adams and Adams-Novikov
spectral sequences}

The differentials in the motivic Adams-Novikov spectral sequence can
be deduced by using the comparison map (over the complex numbers) to
the classical one.  For instance, classically we have
$d_3(\alpha_3)=\alpha_1^4$.  Motivically this doesn't make sense,
because $\alpha_3$ lies in weight $3$ whereas $\alpha_1^4$ lies in
weight $4$.  But it allows us to deduce the motivic differential
$d_3(\alpha_3)=\tau \alpha_1^4$.  In fact, this reasoning shows
\[ d_3(\alpha_{4k-1})=\tau h_1^3\alpha_{4k-3} \quad\text{and}\quad
 d_3(\alpha_{4k+2})=\tau h_1^3\alpha_{4k} \ \ \ (k\geq 1) .
\]
These differentials are depicted in our chart, although except for the
first of them they are not given their true length (to make the
chart easier to read).

It is easy to continue this analysis and deduce all the motivic
differentials in the given range (there are only three others).  This
is done in \cite{HKO}.  The resulting $E_\infty$
term is shown in Appendix~\ref{se:AN}.

Our goal here is really to compare the information given by the
motivic Adams-Novikov spectral sequence to that given by the motivic
Adams spectral sequence.  Recall that these are converging to the same
groups, since $S^{\wedge}_{\{BPL,H\}}\he S^{\wedge}_H$.

Comparing the two spectral sequences, we see that the groups they are
converging to are in complete agreement through the $21$-stem.  
An interesting difference occurs in the $22$-stem, however.  To
make what we are about to say easier to follow, let us use the term
``bound'' to describe classes which are part of an $h_1$-tower
which survives to $E_\infty$: e.g., the towers determined by $h_1$, $P^kh_1$ and
$P^kc_0$ in the Adams spectral sequence.   

In the Adams spectral sequence we have three unbound classes surviving
in column $22$: $h_2c_1$ (weight 13), $c_0d_0$ (weight 13),
and $Pd_0$ (weight 12).  In the Adams-Novikov spectral sequence
we only get two unbound
survivors in the $22$-stem, both in weight $13$.  The only way this is
consistent is if there is a multiplicative extension of the form
\[ \tau(c_0d_0)=Pd_0 \]
in the homotopy groups of $S^{\wedge}_H$.

The next possible difference in the two spectral sequences comes in the
$26$-stem, where so far we have not resolved whether the Adams
differential $d_2([h_3g])$ is equal to $h_1^3h_4c_0$.  But the
Adams-Novikov spectral sequence gives that the $26$-stem of
$S^{\wedge}_{\{H,BPL\}}$ consists of $\M_2\oplus \M_2 \oplus
\M_2/(\tau)$ plus bound terms of higher weight, 
where the first two generators are in weights $14$ and
$16$, and the latter is in weight $15$.  The only way this is consistent
with the Adams spectral sequence is if we
indeed
have the Adams differential $d_2([h_3g])=h_1^3h_4c_0$.

Continuing this kind of comparison between the two spectral sequences,
here is a list of what we find:
\begin{enumerate}[(1)]
\item There is a hidden $\tau$-extension between the element
$h_1[h_3g]$ and the element $d_0^2$.
\item The element $h_1^2[h_3g]$ survives the spectral sequence to
determine an element in $\pi_{29,18}(S^{\wedge}_H)$ which is killed by
$\tau^2$ but not killed by $\tau.$
\item Let $x$ be the generator in $\Ext^{8, (39,18)}_A(\M_2,\M_2)$.
Then in the motivic Adams spectral sequence we have
\[ d_2(x)=h_1^2d_0^2, \quad d_3(\tau x)=c_0Pd_0,\quad \text{and}\quad
d_4(\tau^2 x+h_0^7h_5)=P^2d_0.
\]
\end{enumerate}

\subsection{Another application to classical topology}
\label{se:top2}

To close this section we use our motivic techniques to once again
prove a result which is purely about classical algebraic topology.
Recall that in the classical stable homotopy groups of spheres, there
are unique classes in the 2-components of $\pi_8(S)$ and
$\pi_{14}(S)$ which are detected in the Adams spectral sequence by
$c_0$ and $d_0$, respectively.  These elements of the stable homotopy
groups go under the names $\epsilon$ and $\kappa$.  The product
$c_0d_0$ vanishes in the $E_2$-term of the classical Adams spectral
sequence, but this does not tell us that $\epsilon\kappa$ vanishes---it just
says that $\epsilon\kappa$ lives in higher Adams filtration.  We can
use the algebra of the motivic spectral sequences to show that
$\epsilon\kappa$ is in fact nonzero:

\begin{prop}
\label{pr:kappaepsilon}
In the classical stable homotopy groups of spheres, the product
$\epsilon\kappa$ is nonzero and is detected by $Pd_0$ in the Adams
spectral sequence.
\end{prop}

\begin{proof}
There are unique elements $\epsilon_M$ and $\kappa_M$ in
$\pi_{8,5}(S^{\wedge}_H)$ and $\pi_{14,8}(S^{\wedge}_H)$ which are
detected by $c_0$ and $d_0$ in the motivic Adams spectral sequence.
It must be that $\psi(\epsilon_M)=\epsilon$ and
$\psi(\kappa_M)=\kappa$, where $\psi\colon \pi_{p,q}(S^{\wedge}_H)\ra
\pi_p(S)^{\wedge}_2$ is our map from the end of
Section~\ref{se:compare} (because $\psi(\epsilon_M)$ is detected
by $c_0$ in the classical Adams spectral sequence, etc.)  

The product $\epsilon_M\kappa_M$ is detected by $c_0d_0$, which is {\it
nonzero\/} in the motivic Adams spectral sequence.  We discovered in
the last section that there is a hidden $\tau$-extension in
$\pi_{22,*}(S^{\wedge}_H)$,
telling us that $\tau\cdot(\epsilon_M\kappa_M)$ gives the generator in
$\pi_{22,12}(S^{\wedge}_H)$.  But this generator is detected by
$Pd_0$, and so the same remains true upon applying $\psi$.  That is,
$\psi(\tau\cdot \epsilon_M\kappa_M)$ is nonzero and detected by $Pd_0$.
Recall that $\psi$ is a map of ring spectra, and $\psi(\tau)=1$.  So
$\psi(\tau\cdot \epsilon_M\kappa_M)=\epsilon\kappa$, and we are done.
\end{proof}

Again, Proposition~\ref{pr:kappaepsilon} is a known result in
algebraic topology.  It is implicit, for instance, in the homotopy
groups of $\tmf$ as given in
\cite[Prop. 7.2(5)]{Ba} 
(this was explained to us by Mark Behrens).  Still, our motivic proof
is very simple, only depending on the algebraic computation of the
motivic Adams and Adams-Novikov $E_2$-terms and a comparison of
differentials.  



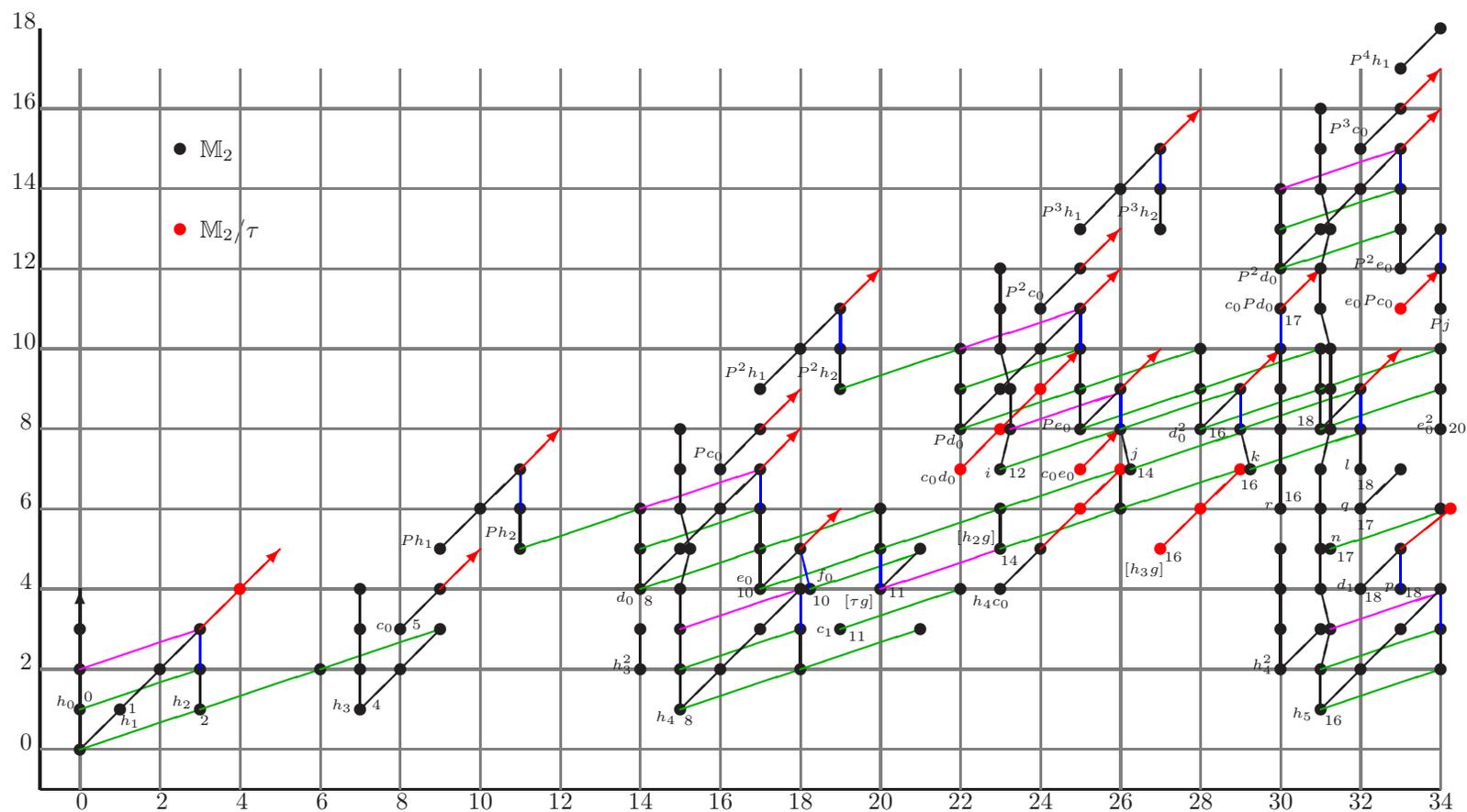
\appendix

\begin{landscape}

\section{The $E_2$-term of the motivic Adams spectral sequence}
\label{se:motivicext}

\setlength{\unitlength}{0.58cm}
\begin{picture}(35,20)(-0.5,-1)

\thicklines

\put(-1,-1){\line(1,0){35}}
\put(-1,-1){\line(0,1){19}}

\put(2.5,15){\bc}
\put(3,14.8){$\M_2$}
\put(2.5,13){\rc}
\put(3,12.8){$\M_2/\tau$}

\put(-0.1,-1.5){0}
\put(1.9,-1.5){2}
\put(3.9,-1.5){4}
\put(5.9,-1.5){6}
\put(7.9,-1.5){8}
\put(9.7,-1.5){10}
\put(11.7,-1.5){12}
\put(13.7,-1.5){14}
\put(15.7,-1.5){16}
\put(17.7,-1.5){18}
\put(19.7,-1.5){20}
\put(21.7,-1.5){22}
\put(23.7,-1.5){24}
\put(25.7,-1.5){26}
\put(27.7,-1.5){28}
\put(29.7,-1.5){30}
\put(31.7,-1.5){32}
\put(33.7,-1.5){34}

\put(-1.5,0){0}
\put(-1.5,2){2}
\put(-1.5,4){4}
\put(-1.5,6){6}
\put(-1.5,8){8}
\put(-1.7,10){10}
\put(-1.7,12){12}
\put(-1.7,14){14}
\put(-1.7,16){16}
\put(-1.7,18){18}

\multiput(0,-1)(2,0){18}{\color[rgb]{0.5,0.5,0.5}\line(0,1){18}}
\multiput(-1,0)(0,2){9}{\color[rgb]{0.5,0.5,0.5}\line(1,0){35}}

\put(0,0){\bc}
\put(0,0){\vl}
\put(0,1){\bc}
\put(-0.6,1.1){\tiny $h_0$}
\put(0.1,1.2){\tiny 0}
\put(0,1){\vl}
\put(0,2){\bc}
\put(0,2){\vl}
\put(0,3){\bc}
\put(0,3){\vector(0,1){1}}

\put(0,0){\dl}
\put(0,0){\htwo}
\put(0,1){\htwo}
\put(0,2){\htwovar}

\put(1,1){\bc}
\put(1,0.6){\tiny $h_1$}
\put(1.2,0.9){\tiny 1}
\put(1,1){\dl}

\put(2,2){\bc}
\put(2,2){\dl}

\put(3,1){\bc}
\put(2.3,1.1){\tiny $h_2$}
\put(3,0.6){\tiny 2}
\put(3,1){\vl}
\put(3,1){\htwo}
\put(3,2){\bc}
\put(3,2){\vlvar}
\put(3,3){\bc}
\put(3,3){\rdl}

\put(4,4){\rc}
\put(4,4){\dltower}

\put(6,2){\bc}
\put(6,2){\htwo}

\put(7,1){\bc}
\put(6.3,1){\tiny $h_3$}
\put(7.3,1){\tiny 4}
\put(7,1){\vl}
\put(7,1){\dl}
\put(7,2){\bc}
\put(7,2){\vl}
\put(7,3){\bc}
\put(7,3){\vl}
\put(7,4){\bc}

\put(8,2){\bc}
\put(8,2){\dl}
\put(8,3){\bc}
\put(7.4,3){\tiny $c_0$}
\put(8.3,3){\tiny 5}
\put(8,3){\dl}

\put(9,3){\bc}
\put(9,4){\bc}
\put(9,4){\dltower}
\put(9,5){\bc}
\put(8,5.1){\tiny $Ph_1$}
\put(9,5){\dl}

\put(10,6){\bc}
\put(10,6){\dl}

\put(11,5){\bc}
\put(10.1,5.3){\tiny $Ph_2$}
\put(11,5){\vl}
\put(11,6){\bc}
\put(11,6){\vlvar}
\put(11,7){\bc}
\put(11,7){\dltower}

\put(11,5){\htwo}

\put(14,2){\bc}
\put(13.3,2){\tiny $h_3^2$}
\put(14,2){\vl}
\put(14,3){\bc}
\put(14,4){\bc}
\put(13.4,3.7){\tiny $d_0$}
\put(14.1,3.6){\tiny 8}
\put(14,4){\vl}
\put(14,4){\dl}
\put(14,5){\bc}
\put(14,5){\vl}
\put(14,6){\bc}

\put(14,4){\htwo}
\put(14,5){\htwo}
\put(14,6){\htwovar}

\put(15,1){\bc}
\put(14.4,0.7){\tiny $h_4$}
\put(15.1,0.6){\tiny 8}
\put(15,1){\vl}
\put(15,1){\dl}
\put(15,2){\bc}
\put(15,2){\vl}
\put(15,3){\bc}
\put(15,3){\vl}
\put(15,4){\bc}
\put(15,4){\vlr}
\put(15,5){\bc}
\put(15,5){\dl}
\put(15.25,5){\bc}
\put(15.25,5){\vll}
\put(15,6){\bc}
\put(15,6){\vl}
\put(15,7){\bc}
\put(15,7){\vl}
\put(15,8){\bc}

\put(15,1){\htwo}
\put(15,2){\htwo}
\put(15,3){\htwovar}

\put(16,2){\bc}
\put(16,2){\dl}
\put(16,6){\bc}
\put(16,6){\dl}
\put(16,7){\bc}
\put(15.3,7.3){\tiny $P c_0$}
\put(16,7){\dl}

\put(17,3){\bc}
\put(17,3){\dl}
\put(17,4){\bc}
\put(16.4,4.2){\tiny $e_0$}
\put(16.4,3.8){\tiny 10}
\put(17,4){\vl}
\put(17,4){\dl}
\put(17,5){\bc}
\put(17,5){\vl}
\put(17,6){\bc}
\put(17,6){\vlvar}
\put(17,7){\bc}
\put(17,7){\dltower}
\put(17,8){\bc}
\put(17,8){\dltower}
\put(17,9){\bc}
\put(16.1,9.3){\tiny $P^2 h_1$}
\put(17,9){\dl}

\put(17,4){\htwo}
\put(17,5){\htwo}

\put(18,2){\bc}
\put(18,2){\vl}
\put(18,3){\bc}
\put(18,3){\vlvar}
\put(18,4){\bc}
\put(18.25,4){\bc}
\put(18.25,4){\vllvar}
\put(18.4,4.25){\tiny $f_0$}
\put(18.3,3.6){\tiny 10}
\put(18,5){\bc}
\put(18,5){\dltower}
\put(18,10){\bc}
\put(18,10){\dl}

\put(18,2){\htwo}
\put(18.25,4){\color[rgb]{0,0.7,0} \line(3,1){2.75}}

\put(19,3){\bc}
\put(18.4,2.9){\tiny $c_1$}
\put(19.2,2.7){\tiny 11}
\put(19,9){\bc}
\put(17.9,9.3){\tiny $P^2 h_2$}
\put(19,9){\vl}
\put(19,10){\bc}
\put(19,10){\vlvar}
\put(19,11){\bc}
\put(19,11){\dltower}

\put(19,3){\htwo}
\put(19,9){\htwo}

\put(20,4){\bc}
\put(19.1,3.6){\tiny $[\tau g]$}
\put(20.2,3.8){\tiny 11}
\put(20,4){\vlvar}
\put(20,4){\dl}
\put(20,5){\bc}
\put(20,5){\vl}
\put(20,6){\bc}

\put(20,4){\htwovar}
\put(20,5){\htwo}

\put(21,3){\bc}
\put(21,5){\bc}

\put(22,4){\bc}
\put(22,7){\rc}
\put(22,7){\rdl}
\put(21,6.7){\tiny $c_0 d_0$}
\put(22,8){\bc}
\put(22,8){\vl}
\put(21.3,7.6){\tiny $P d_0$}
\put(22,8){\dl}
\put(22,8){\htwo}
\put(22,9){\bc}
\put(22,9){\vl}
\put(22,9){\htwo}
\put(22,10){\bc}
\put(22,10){\htwovar}

\put(23,4){\bc}
\put(22.3,3.6){\tiny $h_4 c_0$}
\put(23,4){\dl}
\put(23,5){\bc}
\put(23,5){\vl}
\put(23,5){\htwo}
\put(21.9,5.2){\tiny $[h_2 g]$}
\put(23,4.6){\tiny 14}
\put(23,6){\bc}
\put(23,6){\htwo}
\put(23,7){\bc}
\put(22.6,6.8){\tiny $i$}
\put(23.2,6.8){\tiny 12}
\put(23,7){\vlr}
\put(23,7){\htwo}
\put(23,8){\rc}
\put(23,8){\rdl}
\put(23.25,8){\bc}
\put(23.25,8){\vl}
\put(23.25,8){\htwoshortvar}
\put(23,9){\bc}
\put(23,9){\dl}
\put(23.25,9){\bc}
\put(23.25,9){\vll}
\put(23,10){\bc}
\put(23,10){\vl}
\put(23,11){\bc}
\put(23,11){\vl}
\put(23,12){\bc}

\put(24,5){\bc}
\put(24,5){\rdl}
\put(24,9){\rc}
\put(24,9){\dltower}
\put(24,10){\bc}
\put(24,10){\dl}
\put(24,11){\bc}
\put(23.1,11.3){\tiny $P^2 c_0$}
\put(24,11){\dl}

\put(25,6){\rc}
\put(25,6){\rdl}
\put(25,7){\rc}
\put(24,6.8){\tiny $c_0 e_0$}
\put(25,7){\dltower}
\put(25,8){\bc}
\put(24,8){\tiny $P e_0$}
\put(25,8){\htwo}
\put(25,8){\vl}
\put(25,8){\dl}
\put(25,9){\bc}
\put(25,9){\vl}
\put(25,9){\htwo}
\put(25,10){\bc}
\put(25,10){\vlvar}
\put(25,11){\bc}
\put(25,11){\dltower}
\put(25,12){\bc}
\put(25,12){\dltower}
\put(25,13){\bc}
\put(24,13.3){\tiny $P^3 h_1$}
\put(25,13){\dl}

\put(26,6){\bc}
\put(26,6){\vl}
\put(26,6){\htwo}
\put(26,7){\rc}
\put(26.25,7){\bc}
\put(26.25,7){\vll}
\put(26.25,7.3){\tiny $j$}
\put(26.4,6.8){\tiny 14}
\put(26.25,7){\htwoshort}
\put(26,8){\bc}
\put(26,8){\vlvar}
\put(26,8){\htwo}
\put(26,9){\bc}
\put(26,9){\dltower}
\put(26,14){\bc}
\put(26,14){\dl}

\put(27,5){\rc}
\put(26.1,4.4){\tiny $[h_3 g]$}
\put(27.1,4.7){\tiny 16}
\put(27,5){\rdl}
\put(27,13){\bc}
\put(25.9,13.3){\tiny $P^3 h_2$}
\put(27,13){\vl}
\put(27,14){\bc}
\put(27,14){\vlvar}
\put(27,15){\bc}
\put(27,15){\dltower}

\put(28,6){\rc}
\put(28,6){\rdl}
\put(28,8){\bc}
\put(27.2,7.8){\tiny $d_0^2$}
\put(28.2,7.8){\tiny 16}
\put(28,8){\vl}
\put(28,8){\dl}
\put(28,8){\htwo}
\put(28,9){\bc}
\put(28,9){\vl}
\put(28,9){\htwo}
\put(28,10){\bc}

\put(29,7){\rc}
\put(29.25,7){\bc}
\put(29.25,7.2){\tiny $k$}
\put(29,6.5){\tiny 16}
\put(29.25,7){\vll}
\put(29.25,7){\htwoshort}
\put(29,8){\bc}
\put(29,8){\vlvar}
\put(29,8){\htwo}
\put(29,9){\bc}
\put(29,9){\dltower}

\put(30,2){\bc}
\put(29.3,2){\tiny $h_4^2$}
\put(30,2){\vl}
\put(30,2){\dl}
\put(30,3){\bc}
\put(30,3){\vl}
\put(30,4){\bc}
\put(30,4){\vl}
\put(30,5){\bc}
\put(30,6){\bc}
\put(29.6,6){\tiny $r$}
\put(30.1,6.2){\tiny 16}
\put(30,6){\vl}
\put(30,7){\bc}
\put(30,7){\vl}
\put(30,8){\bc}
\put(30,8){\vl}
\put(30,9){\bc}
\put(30,9){\vl}
\put(30,10){\bc}
\put(30,10){\vlvar}
\put(30,11){\bc}
\put(28.6,11){\tiny $c_0 Pd_0$}
\put(30.1,10.6){\tiny 17}
\put(30,11){\dltower}
\put(30,12){\bc}
\put(28.9,11.7){\tiny $P^2 d_0$}
\put(30,12){\vl}
\put(30,12){\dl}
\put(30,12){\htwo}
\put(30,13){\bc}
\put(30,13){\vl}
\put(30,13){\htwo}
\put(30,14){\bc}
\put(30,14){\htwovar}

\put(31,1){\bc}
\put(30.3,0.8){\tiny $h_5$}
\put(31.1,0.6){\tiny 16}
\put(31,1){\dl}
\put(31,1){\vl}
\put(31,1){\htwo}
\put(31,2){\bc}
\put(31,2){\vlr}
\put(31,2){\htwo}
\put(31,3){\bc}
\put(31.25,3){\bc}
\put(31.25,3){\vll}
\put(31.25,3){\htwoshortvar}
\put(31,4){\bc}
\put(31,4){\vl}
\put(31,5){\bc}
\put(31,5){\vl}
\put(31.25,5){\bc}
\put(31.25,5.2){\tiny $n$}
\put(31.4,4.7){\tiny 17}
\put(31.25,5){\htwo}
\put(31,6){\bc}
\put(31,6){\vl}
\put(31,7){\bc}
\put(31,7){\vlr}
\put(31,8){\bc}
\put(30.4,8.1){\tiny 18}
\put(31,8){\vl}
\put(31,8){\dl}
\put(31,8){\htwo}
\put(31,9){\bc}
\put(31,9){\vl}
\put(31,9){\htwo}
\put(31,10){\bc}
\put(31.25,8){\bc}
\put(31.25,8){\vl}
\put(31.25,9){\bc}
\put(31.25,9){\vl}
\put(31.25,10){\bc}
\put(31.25,10){\vll}
\put(31,11){\bc}
\put(31,11){\vl}
\put(31,12){\bc}
\put(31,12){\vlr}
\put(31,13){\bc}
\put(31,13){\dl}
\put(31.25,13){\bc}
\put(31.25,13){\vll}
\put(31,14){\bc}
\put(31,14){\vl}
\put(31,15){\bc}
\put(31,15){\vl}
\put(31,16){\bc}

\put(32,2){\bc}
\put(32,2){\dl}
\put(32,4){\bc}
\put(31.4,4){\tiny $d_1$}
\put(32.1,3.7){\tiny 18}
\put(32,4){\dl}
\put(32,6){\bc}
\put(31.5,6){\tiny $q$}
\put(31.9,5.5){\tiny 17}
\put(32,6){\dl}
\put(32,7){\bc}
\put(31.6,7){\tiny $l$}
\put(31.9,6.5){\tiny 18}
\put(32,7){\vl}
\put(32,8){\bc}
\put(32,8){\vlvar}
\put(32,9){\bc}
\put(32,9){\dltower}
\put(32,14){\bc}
\put(32,14){\dl}
\put(32,15){\bc}
\put(31.2,15.4){\tiny $P^3 c_0$}
\put(32,15){\dl}

\put(33,3){\bc}
\put(33,3){\dl}
\put(33,4){\bc}
\put(32.6,4){\tiny $p$}
\put(33.1,3.8){\tiny 18}
\put(33,4){\vlvar}
\put(33,5){\bc}
\put(33,5){\rdlr}
\put(33,7){\bc}
\put(33,11){\rc}
\put(31.6,11.1){\tiny $e_0 P c_0$}
\put(33,11){\dltower}
\put(33,12){\bc}
\put(31.8,12){\tiny $P^2 e_0$}
\put(33,12){\vl}
\put(33,12){\dl}
\put(33,13){\bc}
\put(33,13){\vl}
\put(33,14){\bc}
\put(33,14){\vlvar}
\put(33,15){\bc}
\put(33,15){\dltower}
\put(33,16){\bc}
\put(33,16){\dltower}
\put(33,17){\bc}
\put(31.7,17.1){\tiny $P^4 h_1$}
\put(33,17){\dl}

\put(34,2){\bc}
\put(34,2){\vl}
\put(34,3){\bc}
\put(34,3){\vlvar}
\put(34,4){\bc}
\put(34,6){\bc}
\put(34.25,6){\rc}
\put(34,8){\bc}
\put(33.4,8){\tiny $e_0^2$}
\put(34.2,7.9){\tiny 20}
\put(34,8){\vl}
\put(34,9){\bc}
\put(34,9){\vl}
\put(34,10){\bc}
\put(34,11){\bc}
\put(33.7,10.5){\tiny $Pj$}
\put(34,11){\vl}
\put(34,12){\bc}
\put(34,12){\vlvar}
\put(34,13){\bc}
\put(34,18){\bc}









\end{picture}

\section{The motivic Adams-Novikov spectral sequence}
\label{se:AN}

\setlength{\unitlength}{0.58cm}
\begin{picture}(35,20)(-0.5,-1)

\thicklines
\put(-0.75,7){$E_2=E_3$}
\put(31,18.4){$E_6=E_\infty$}
\put(0.5,18.6){\bc}
\put(1,18.4){\small{$\F_2[\tau]$}}
\put(4,18.6){\bbox}
\put(4.5,18.4){\small{$\Z/n[\tau]$}}

\put(7.5,18.6){\rc}
\put(8,18.4){\small{$\F_2[\tau]/(\tau)$}}
\put(11.5,18.6){\smallrc}
\put(11.5,18.6){\bigrc}
\put(12,18.4){\small{$\F_2[\tau]/(\tau^2)$}}

\put(-1,-1){\line(1,0){35}}
\put(-1,-1){\line(0,1){9}}
\put(-1,9){\line(0,1){9}}
\put(-1,9){\line(1,0){35}}

\put(-0.1,-1.5){0}
\put(1.9,-1.5){2}
\put(3.9,-1.5){4}
\put(5.9,-1.5){6}
\put(7.9,-1.5){8}
\put(9.7,-1.5){10}
\put(11.7,-1.5){12}
\put(13.7,-1.5){14}
\put(15.7,-1.5){16}
\put(17.7,-1.5){18}
\put(19.7,-1.5){20}
\put(21.7,-1.5){22}
\put(23.7,-1.5){24}
\put(25.7,-1.5){26}
\put(27.7,-1.5){28}
\put(29.7,-1.5){30}
\put(31.7,-1.5){32}
\put(33.7,-1.5){34}

\put(-1.5,-0.25){0}
\put(-1.5,1.75){2}
\put(-1.5,3.75){4}
\put(-1.5,5.75){6}
\put(-1.5,7.75){8}
\put(-1.7,9.75){0}
\put(-1.7,11.75){2}
\put(-1.7,13.75){4}
\put(-1.7,15.75){6}
\put(-1.7,17.75){8}

\multiput(0,-1)(2,0){18}{\color[rgb]{0.5,0.5,0.5}\line(0,1){9}}
\multiput(0,9)(2,0){18}{\color[rgb]{0.5,0.5,0.5}\line(0,1){9}}
\multiput(-1,0)(0,2){10}{\color[rgb]{0.5,0.5,0.5}\line(1,0){35}}

\put(0,0){\bbox}
\put(0,0){\dl}
\put(1,1){\bc}

\multiput(3,1)(4,0){8}{\bbox}
\multiput(1,1)(4,0){9}{\bc}
\multiput(5,1)(4,0){8}{\da}
\multiput(7,1)(4,0){7}{\da}

\multiput(1,1)(1,1){5}{\bc}
\multiput(1,1)(1,1){4}{\dl}

\put(5,5){\da}

\put(6,2){\bc}
\put(6,2){\htwo}
\put(5.0,2.4){\tiny{$\beta_{2/2}$}}
\put(8,2){\bc}
\put(8,2){\dl}
\put(8.2,1.7){\tiny{$\beta_2$}}
\put(9,3){\bc}
\put(14,2){\bc}
\put(14,2.3){\bc}
\put(14,2){\dl}
\put(14,2){\htwoshort}
\put(14.2,1.7){\tiny{$\beta_3$}}
\put(15,3){\bc}

\put(16,2){\bc}
\put(16,2){\dl}
\put(16.6,3){\bc}
\put(16.6,3){\htwo}
\put(17,3){\bc}
\put(17,3){\dl}
\put(18,4){\bc}
\put(18,2){\bbox}
\put(18,2){\htwoshort}
\put(19,3){\bc}
\put(19,3){\htwoshort}
\put(20,2){\bc}
\put(20,2){\dl}
\put(20,4){\bbox}
\put(20,4){\htwo}
\put(20.6,2.9){\bc}
\put(21,3){\bc}
\put(21,3){\dl}
\put(21.6,3.9){\bc}
\put(22,4){\bc}
\put(22,4){\dl}
\put(23,5){\bbox}
\put(23,5){\htwo}

\put(23,3){\bc}
\put(23,3){\dl}
\put(24,4){\bc}
\put(24,4){\dl}
\put(25,5){\bc}
\put(26,2){\bc}
\put(26,6){\bc}
\put(26,6){\htwo}
\put(28,6){\bc}
\put(28,6){\dl}
\put(29,7){\bc}
\put(30,2){\bc}
\put(30,2){\dl}
\put(30,2.3){\bc}
\put(31,3){\bc}
\put(31,3.3){\bc}
\put(31,3.3){\htwo}
\put(32,2){\bc}
\put(32,2){\dl}
\put(32,2.4){\bc}
\put(32,2.4){\dl}
\put(32,4){\bc}
\put(33,3){\bc}
\put(33,3){\dl}
\put(33,3.4){\bc}
\put(33,2.6){\bc}
\put(33,5){\bc}
\put(32,4){\dl}
\put(33,5){\dl}
\put(34,2){\bbox}
\put(34,4.4){\bc}
\put(34,4){\bc}
\put(34,6){\bc}
\put(34,6.4){\bc}

\put(5,1){\dthree}
\multiput(13,1)(8,0){3}{\tinydthree}
\multiput(11,1)(8,0){3}{\tinydthree}
\put(30,2.6){\dfive}
\put(26,2){\dthree}
\put(34,2){\dthree}

\put(1,0.5){\tiny{$\alpha_1$}}
\put(5,0.5){\tiny{$\alpha_3$}}
\put(9,0.5){\tiny{$\alpha_5$}}
\put(13,0.5){\tiny{$\alpha_7$}}
\put(3,0.5){\tiny{$\alpha_2$}}
\put(7,0.5){\tiny{$\alpha_4$}}
\put(11,0.5){\tiny{$\alpha_6$}}
\put(15,0.5){\tiny{$\alpha_8$}}
\put(23,2.5){\tiny{$\eta_{3/2}$}}
\put(31.2,4.3){\tiny{$x_{32}$}}
\put(25.3,2.2){\tiny{$\beta_5$}}
\put(28.9,2.5){\tiny{$\beta_{6/2}$}}
\put(27.1,5.5){\tiny{$P\beta_2$}}


\put(0,10){\bbox}
\multiput(1,11)(1,1){3}{\bc}
\multiput(0,10)(1,1){3}{\dl}
\multiput(4,14)(1,1){2}{\rc}
\put(5,15){\dltower}
\multiput(3,13)(1,1){2}{\rdl}

\multiput(3,11)(8,0){2}{\bbox}
\put(7.4,11){\bbox}
\multiput(19,11)(4,0){3}{\bbox}
\put(31.4,11){\bbox}
\put(15.4,11){\bbox}

\put(6,12){\bc}
\put(6,12){\htwo}
\multiput(7.4,11)(1,1){2}{\dl}
\multiput(8.4,12)(1,1){2}{\bc}
\put(9.4,13){\dltower}
\multiput(9,11)(1,1){3}{\bc}
\multiput(9,11)(1,1){2}{\dl}
\put(11,13){\dltower}
\put(8,12){\bc}
\put(8,12){\dl}
\put(9,13){\bc}
\put(14,12){\bc}
\put(14,12.4){\bc}
\put(15,13){\bc}
\put(14,12){\dl}

\multiput(15.4,11)(1,1){2}{\dl}
\multiput(16.4,12)(1,1){2}{\bc}
\put(17.4,13){\dltower}
\multiput(23,11)(1,1){2}{\dl}
\multiput(24,12)(1,1){2}{\bc}
\put(25,13){\dltower}

\multiput(31.4,11)(1,1){2}{\dl}
\multiput(32.4,12)(1,1){2}{\bc}
\put(33.4,13){\dltower}

\multiput(17.4,11)(1,1){3}{\bc}
\multiput(17.4,11)(1,1){2}{\dl}
\put(19.4,13){\dltower}
\multiput(25,11)(1,1){3}{\bc}
\multiput(25,11)(1,1){2}{\dl}
\put(27,13){\dltower}

\put(14,12.05){\htwodshort}

\put(16,12){\bc}
\put(16,12){\dl}
\put(16.6,13){\bc}
\put(16.8,13.1){\htwo}
\put(17,13){\bc}
\put(17,13){\dl}
\put(18,14){\bc}
\put(18,12){\bbox}
\put(18,12.1){\htwodshort}
\put(19,13){\bc}
\put(19,13.1){\htwodshort}
\put(20,12){\bc}
\put(20,12){\dl}
\put(20,14){\bbox}
\put(20,14){\htwo}
\put(20.6,12.9){\bc}
\put(21,13){\bc}
\put(21,13){\dl}
\put(21.6,13.9){\bc}
\put(22,14){\bc}
\put(22,14){\dl}
\put(23,15){\bbox}
\put(23,15){\htwo}

\put(23,13){\bc}
\put(23,13){\dl}
\put(24,14){\bc}
\put(24,14){\rdl}
\put(25,15){\rc}
\put(26,12){\bc}
\put(26,16){\bc}
\put(26,16){\htwo}
\put(28,16){\bc}
\put(28,16){\rdl}
\put(29,17){\smallrc}
\put(29,17){\bigrc}
\put(30,12){\bc}
\put(30,12){\dl}
\put(31,13){\bc}
\put(31,13.4){\bc}
\put(31,13.4){\htwo}
\put(32,12){\bc}
\put(32,12){\dl}
\put(32,12.4){\bc}
\put(32,12.4){\dl}
\put(32,14){\bc}
\put(33,13){\bc}
\put(33,13){\dl}
\put(33,13.4){\bc}
\put(33.4,12.6){\bc}
\put(33,15){\rc}
\put(32,14){\rdl}
\put(33,15){\rdl}
\put(34,12){\bc}
\put(34,14.4){\bc}
\put(34,14){\bc}
\put(34,16){\rc}
\put(34,16.4){\bc}

\end{picture}

\section{The $E_4$-term of the motivic May spectral sequence}
\label{fig:mayss}

\par\noindent

\setlength{\unitlength}{0.8cm}
\begin{picture}(21,17)(-0.5,-1)

\thicklines

\put(-1,-1){\line(1,0){21}}
\put(-1,-1){\line(0,1){13}}

\put(-0.1,-1.5){0}
\put(1.9,-1.5){2}
\put(3.9,-1.5){4}
\put(5.9,-1.5){6}
\put(7.9,-1.5){8}
\put(9.7,-1.5){10}
\put(11.7,-1.5){12}
\put(13.7,-1.5){14}
\put(15.7,-1.5){16}
\put(17.7,-1.5){18}
\put(19.7,-1.5){20}

\put(-1.5,0){0}
\put(-1.5,2){2}
\put(-1.5,4){4}
\put(-1.5,6){6}
\put(-1.5,8){8}
\put(-1.7,10){10}
\put(-1.7,12){12}

\multiput(0,-1)(2,0){11}{\color[rgb]{0.5,0.5,0.5}\line(0,1){13}}
\multiput(-1,0)(0,2){7}{\color[rgb]{0.5,0.5,0.5}\line(1,0){21}}

\multiput(0,0)(0,1){4}{\bc}
\multiput(0,0)(0,1){3}{\vl}
\put(0,3){\vltower}
\multiput(1,1)(1,1){3}{\bc}
\multiput(0,0)(1,1){3}{\dl}
\put(0,0){\htwo}
\put(0,1){\htwo}
\put(0,2){\htwovar}
\put(3,3){\dltower}
\put(1.2,0.8){\tiny $h_1$}
\put(-0.7,0.8){\tiny $h_0$}

\multiput(3,1)(0,1){2}{\bc}
\put(3,1){\vl}
\put(3,1){\htwo}
\put(3,2){\vlvar}
\put(3,0.5){\tiny $h_2$}
\put(6,2){\bc}
\put(6,2){\htwo}

\multiput(7,1)(0,1){5}{\bc}
\put(7,0.5){\tiny $h_3$}
\multiput(7,1)(0,1){4}{\vl}
\put(7,5){\vltower}
\multiput(8,2)(1,1){2}{\bc}
\multiput(7,1)(1,1){2}{\dl}
\multiput(8,3)(1,1){2}{\bc}
\put(7.3,3){\tiny $c_0$}
\put(8,3){\dl}
\put(9,4){\dltower}

\multiput(8,4)(0,1){3}{\bc}
\put(7.3,3.5){\tiny $b_{20}^2$}
\put(8,4){\vl}
\put(8,5){\vl}
\put(8,4){\htwo}
\put(8,5){\htwo}
\put(8,6){\htwovar}
\put(8,6){\vltower}
\multiput(9,5)(1,1){3}{\bc}
\multiput(8,4)(1,1){3}{\dl}
\multiput(11,5)(0,1){2}{\bc}
\put(11,5){\vl}
\put(11,6){\vlvar}
\put(11,7){\dltower}

\multiput(13.7,2)(0,1){3}{\bc}
\put(13.4,1.4){\tiny $h_3^2$}
\put(13.7,2){\vl}
\put(13.7,3){\vl}
\put(13.7,4){\vltower}
\multiput(14,4)(0,1){3}{\bc}
\put(14.1,3.7){\tiny $d_0$}
\multiput(14,4)(1,1){4}{\bc}
\multiput(14,4)(1,1){3}{\dl}
\multiput(14,4)(0,1){2}{\vl}
\put(14,4){\htwo}
\put(14,5){\htwo}
\put(14,6){\htwovar}
\put(17,7){\dltower}

\multiput(15,1)(0,1){3}{\bc}
\put(15,0.5){\tiny $h_4$}
\multiput(15,1)(1,1){4}{\bc}
\put(15,1){\vl}
\put(15,2){\vl}
\put(15,3){\vltower}
\put(15,1){\htwo}
\put(15,2){\htwo}
\put(15,3){\htwovar}
\multiput(15,1)(1,1){3}{\dl}
\put(18,4){\dl}
\put(19,5){\rc}
\put(19,5){\dltower}
\multiput(18,2)(0,1){3}{\bc}
\put(18,2){\vl}
\put(18,3){\vlvar}

\put(14.7,3){\bc}
\put(14.7,3){\vltower}
\put(13.8,2.6){\tiny $h_2 b_{30}$}
\multiput(17,4)(0,1){3}{\bc}
\multiput(17,4)(0,1){2}{\vl}
\put(16.5,4.2){\tiny $e_0$}
\put(17,4){\htwo}
\put(17,5){\htwo}
\put(17,6){\vlvar}
\put(17,4){\dl}
\put(18,5){\bc}
\put(18,5){\dltower}

\put(16,7){\bc}
\put(14.7,7){\tiny $b_{20}^2 c_0$}
\put(16,7){\dl}
\put(17,8){\bc}
\put(17,8){\dltower}

\multiput(16,8)(0,1){3}{\bc}
\put(15.2,8.3){\tiny $b_{20}^4$}
\multiput(16,8)(1,1){4}{\bc}
\multiput(16,8)(1,1){3}{\dl}
\put(16,8){\vl}
\put(16,9){\vl}
\put(16,10){\vltower}
\put(16,8){\htwo}
\put(16,9){\htwo}
\put(16,10){\htwovar}

\multiput(19,9)(0,1){2}{\bc}
\put(19,9){\vl}
\put(19,10){\vlvar}
\put(19,11){\dltower}

\put(18.5,4){\bc}
\put(18.5,3.5){\tiny $f_0$}
\put(18.5,4){{\color[rgb]{0,0,1}\line(-1,2){0.5}}}
\put(19,3){\bc}
\put(19,2.5){\tiny $c_1$}

\put(20,4){\bc}
\put(19.6,3.6){\tiny $g$}
\put(20,4){\vl}
\put(20,5){\bc}
\put(20,5){\vl}
\put(20,6){\bc}

\put(8,4){\done}
\put(14.7,3){\done}
\put(20,4){\done}

\put(2.5,15.5){\bc}
\put(3,15.3){$\M_2$}
\put(2.5,13.5){\rc}
\put(3,13.3){$\M_2/\tau$}

\put(6.2,15.3){$c_0 = h_1 h_0(1)$}
\put(6.2,14.3){$d_0 = b_{20} b_{21} + h_1^2 b_{30}$}
\put(6.2,13.3){$e_0 = b_{21} h_0(1)$}

\put(12.2,15.3){$f_0 = h_2^2 b_{30}$}
\put(12.2,14.3){$c_1 = h_2 h_1(1)$}
\put(12.2,13.3){$g = b_{21}^2$}
\end{picture}

\end{landscape}


\bibliographystyle{amsalpha}

\end{document}